\documentclass{article}

\usepackage{amsmath,amsthm,amsfonts,a4,bezier}
\usepackage[numeric,initials,nobysame,abbrev]{amsrefs}
\usepackage{amssymb}
\usepackage{amscd}
\usepackage{graphpap}
\usepackage[pdftex]{color,graphicx,xcolor}
\usepackage{enumerate}
\usepackage{setspace}
\usepackage{comment}
\usepackage{mathtools}
\usepackage{relsize}
\usepackage{indentfirst}
\usepackage{verbatim,bbm}
\usepackage{extarrows}   
\usepackage{centernot}
\usepackage{tikz}
\usepackage{titling}
\usepackage{ulem}
\usepackage[colorlinks=true, allcolors=blue]{hyperref} 
\usepackage{tzplot}
\usepackage{authblk}
\usepackage{enumitem}

\newtheorem{theorem}{Theorem}

\newtheorem{definition}{Definition}

\newtheorem{proposition}{Proposition}
\newtheorem{lemma}{Lemma}

\newtheorem{remark}{Remark}
\newtheorem{Conjecture}{Conjecture}

\newcommand{\E}{\mathcal{E}}

\newcommand{\Z}{{\mathbb{Z}}}
\newcommand{\LL}{{\mathbb{L}}}
\newcommand{\N}{{\mathbb{N}}}

\newcommand{\x}{\mathbf{x}}
\newcommand{\y}{\mathbf{y}}
\newcommand{\z}{{\mathbf {z}}}
\newcommand{\eps}{\varepsilon}

\newcommand{\G}{{\mathcal{G}}}

\renewcommand{\leq}{\leqslant}
\renewcommand{\geq}{\geqslant}

\renewcommand{\P}{{\mathbb{P}}}
\renewcommand{\u}{\mathbf{u}}
\renewcommand{\v}{\mathbf{v}}
\renewcommand{\o}{{\omega}}
\renewcommand{\t}{{\mathbf {t}}}
\renewcommand{\u}{{\mathbf {u}}}

\begin{document}

\title{\bf{Truncated long-range percolation of words on the square lattice}}
\author[1]{Pablo A. Gomes}
\author[2]{Otávio Lima}
\author[3]{Roger W. C. Silva}
\affil[1]{\small Universidade de São Paulo, Instituto de Matemática e Estatística, Rua do Matão 1010, São Paulo, Brazil
\texttt{pagomes@usp.br}}
\affil[2]{\small Universidade Federal de Minas Gerais, Departamento de Estatística,  Av. Antônio Carlos 6627, Belo Horizonte, Brazil \texttt{otaviooasl@ufmg.br}}
\affil[3]{\small Universidade Federal de Minas Gerais, Departamento de Estatística,  Av. Antônio Carlos 6627, Belo Horizonte, Brazil \texttt{rogerwcs@est.ufmg.br}}

\date{}

\maketitle

\begin{abstract}

We study mixed long-range percolation on the square lattice. Each vertical edge of unit length is independently open with probability $\varepsilon$, and each horizontal edge of length $i$ is independently open with probability $p_i$. Also, each vertex is independently assigned a random variable taking values 1 and 0 with probabilities $p$ and $1-p$, respectively. We prove that, for a broad class of anisotropic long-range percolation models satisfying suitable regularity conditions on the sequence $p_i$, all words (semi-infinite binary sequences) are seen simultaneously from the origin with positive probability, even if all edges with length larger than some constant (depending on $\varepsilon$, $p$, and on the sequence $(p_i)_i$) are suppressed.

\medskip


\end{abstract}

\onehalfspacing

\section{Introduction}\label{intro}

We study long-range percolation on the square lattice. More precisely, let $\G=(\Z^2,\E_V\cup \E_H)$, where $\E_V$ denotes the set of vertical edges of unit length, and $\E_H$ denotes the set of horizontal edges of all lengths. Specifically, we define
\begin{equation}\label{vertical}
\E_{V}=\left\{\{(x,y_1),(x,y_2)\}\in\Z^2\times\Z^2: |y_1-y_2|=1\right\},
\end{equation}
\begin{equation}\label{horizontal}
\E_{H,i}=\left\{\{(x_1,y),(x_2,y)\}\in\Z^2\times\Z^2: |x_1-x_2|=i\right\}.
\end{equation}
We then write $\E_H=\cup_{i=1}^{\infty}\E_{H,i}$ and $\E=\E_V\cup\E_H$. 

The percolation process is defined as follows: given a sequence $(p_i)_{i\in\N}$, $p_i\in [0,1]$, we consider the probability space $(\Omega,\mathcal{F},\P)$, where $\Omega=\{0,1\}^{\E}$ and $\mathcal{F}$ is the $\sigma$-field generated by the cylinder sets of $\Omega$. A configuration $\omega\in\{0,1\}^{\E}$ is assigned to the edges, and for each $e\in\E$, we write $\omega(e)$ for its value at $e$. If $\omega(e)=1$ ($\omega(e)=0$), we say $e$ is open (closed). The probability measure is given by 
\begin{equation*}\P=\prod_{\{ x,y \}\in \E}\mu_{\{ x,y \}},
\end{equation*}
where $\mu_{\{ x,y \}}(\omega(\{ x,y \})=1)=p_{||x-y||}$ is a Bernoulli law. 

Given $K\in\N$, we introduce the $K$-truncated sequence $(p_{K,i})_{i\in\N}$, defined by 
\begin{equation}\label{trun_seq}p_{K,i}=\left\{\begin{array}{ll}
p_i&\mbox{if}\quad i\leq K,\\
 0&\mbox{if}\quad i>K,
\end{array}\right.
\end{equation}
and the corresponding $K$-truncated percolation model $(\Omega,\mathcal{F},\P^K)$, where \begin{equation}\label{truncated}\P^K=\prod_{\{ x,y \}\in \E}\mu_{K,\{ x,y \}},
\end{equation}
with
$\mu_{K,\{ x,y \}}(\omega(\{ x,y \})=1)=p_{K,||x-y||}$.

We say that two vertices $u,v\in\Z^2$ are connected in the configuration $\omega$ if there exists a sequence of distinct vertices $(v_0=u,v_1,\dots,v_n=v)$ in $\Z^2$ such that $\{ v_i,v_{i+1}\}\in\mathcal{E}$ is open in $\omega$ for all $i=0,\dots,n-1$. We write $\{u\leftrightarrow v\}$ for the event that $u$ and $v$ are connected by an open path. The cluster of $v$ in the configuration $\omega$ is defined as $C_v(\omega)=\{u\in\Z^2: v\leftrightarrow u\mbox{ in $\omega$}\}$ (when $v$ is the origin we write simply $C=C(\omega)$ instead of $C_{0}(\omega)$). Percolation is said to occur at the vertex $v$ if $|C_v(\omega)|$ is infinite.

\subsection{Background and motivation}
The classical truncation question asks: if $\P(|C|=\infty)>0$, is it true that $\P^K(|C|=\infty)>0$ for some large $K$? This problem was first considered in \cite{MS}, where the authors provided affirmative answers under the assumption that the sequence $(p_i)_{i\in\N}$ decays exponentially.

A simpler (yet still open) problem is the following: assuming the sum of the $p_i$ diverges, does the truncation question on the square lattice have an affirmative answer? The first significant progress in this direction was made in \cite{SSV}, where the authors analyzed a model with $p_{V}=\varepsilon>0$ and $p_{H,i}=p_i$ for $i\geq 1$. They show that if $p_i\geq (i\log i)^{-1}$ for all sufficiently large $i$, then the truncation question has an affirmative answer. Their argument is based on a renormalization scheme that leads to a dependent percolation process dominating an anisotropic nearest-neighbor Bernoulli percolation model. In this auxiliary model, the probabilities $p_V$ and $p_H$ of vertical and horizontal edges being open satisfy $p_V + p_H > 1$, which guarantees almost sure percolation (see \cite{K}).

Removing the regularity assumption on the sequence $(p_i)_{i\in\N}$, in \cite{FLS} the authors consider a percolation model on $\Z^d$, $d\geq 2$, where all edges of length $i$ parallel to some coordinate axis are independently open with probability $p_i$. They prove that the truncation question has an affirmative answer under the condition that $\limsup p_i>0$. Their proof involves constructing an isomorphism between a subgraph of $\G$ and an $s$-dimensional slab of $\Z^2$ for some $s$ large, followed by an application of the classical result of Grimmett and Marstrand \cite{GM}. For related results on other graphs under the same hypothesis, see \cite{AHLV}.

The truncation question has also been explored under various assumptions on the $p_i$'s; see \cite{LS} and the more recent work \cite{CL}. For higher dimensions ($d\geq 3$), affirmative answers have been provided in \cite{ELV}, \cite{FL}, and more recently in \cite{B}.

The main motivation of this work is to investigate the truncation question in a broader setting, namely percolation of words, which we introduce below.

\subsection{Percolation of words}
The problem of percolation of words was introduced in a pioneering paper by I. Benjamini and H. Kesten \cite{BK} and is formulated as follows: consider a graph $G=(V,E)$ with a countably infinite vertex set $V$. Each vertex $v\in V$ is independently assigned a random variable $\eta(v)$, which takes the value 1 with probability $p$ and 0 with probability $1-p$. This induces the probability space $(H,\mathcal{H},\P_p)$, where $H=\{0,1\}^{V}$, $\mathcal{H}$ is the $\sigma$-field generated by the cylinder sets of $H$, and $\P_p$ is the product measure on $\mathcal{H}$. A typical element of $H$ is denoted by $\eta$, and $\eta(v)$ denotes the state of vertex $v$ in the configuration $\eta$. 

In percolation of words, one is interested in the existence or nonexistence of a self-avoiding sequence of distinct vertices $(v,v_1,v_2,\dots)$ with $\eta(v_i)=\xi_i$, $i\geq 1$, for any prescribed sequence $\xi=\{\xi_i\}_{i\geq 1}\in\{0,1\}^{\N}$. In this case, we say that $\xi$ is seen from vertex $v$ (note that the state of $v$ plays no role here). The main goal is to understand when the collection of sequences $\{\xi_i\}_{i\geq 1}$ that are seen is large. Formally, write
$$\Xi=\{0,1\}^{\mathbb{N}}$$ 
and let $\xi=\{\xi_n\}_{n\in\mathbb{N}}$ be an element of $\Xi$, which we refer to as a \textbf{word}.  For $v\in V$, define the sets
\begin{equation}\label{set_words}W_v=W_v(\eta)=\{\xi\in \Xi: \xi \mbox{ is seen from $v$ in }\eta\},
\end{equation}
$$W_{\infty}=W_{\infty}(\eta)=\bigcup_{v\in V}W_v(\eta).$$ 
Here, $W_{\infty}$ is the set of words seen from some vertex in $G$. Clearly, the largest these sets can be is $\Xi$. 

The study of percolation of words is particularly challenging due to its lack of monotonicity. Unlike Bernoulli percolation (where the focus is on events such as "an infinite connected component of 1’s exists"), the events considered in percolation of words are generally neither increasing nor decreasing. 

Denote by $\LL^d$ the usual hypercubic lattice with nearest neighbors. In \cite{BK}, the authors investigate the problem of percolation of words on $\LL^d_+$
and prove that the event $\{W_{\infty}=\Xi\}$ occurs almost surely when $p=1/2$ and $d\geq 10$. They also show that, when $p=1/2$ and $d\geq 40$, there is a vertex from which one sees all words almost surely, and (by ergodicity) there is a strictly positive probability that one sees all words from the origin.

In \cite{KSZ}, the authors show that \textbf{almost all} words (with respect to the product measure $\nu_{\lambda}:= ((1-\lambda)\delta_0+\lambda\delta_1)^{\otimes\N}$ on $\Xi$, $0<\lambda<1$) can be seen in critical site percolation on the triangular lattice. In a subsequent paper \cite{KSZ2}, the authors examine the problem of seeing all words in site percolation on the closed-packed graph of the square lattice, which is formed by adding diagonal edges to each face of $\LL^2$. They show that for every fixed $p\in(1-p_c(\LL^2),p_c(\LL^2))$, the event $\{\mbox{there exists $v$ such that }W_{v}=\Xi\}$ has probability 1. Here, $p_c(\LL^d)$ denotes the critical threshold for Bernoulli site percolation on $\LL^d$.

 In a recent work \cite{NTT}, the authors show that the event $\{W_{\infty}=\Xi\}$ occurs almost surely on $\LL^d$, for $d\geq 3$ and $p\in(p_c(\LL^d),1-p_c(\LL^d))$, answering the question posed in Open Problem 2 in \cite{BK}. 

 \subsection{Results}

 Previous works address percolation of words on fixed graphs with all edges present. Here, we consider a more general setting: percolation on a random graph, introducing two layers of randomness. First, a long-range bond configuration $\omega$ is sampled according to $\P^K$. Then, the occurrence of a word $\xi$ in $\omega$ is determined by the conditional measure $\P_p^{\omega}$. Formally, writing $\P_{p}^K$ for the probability measure on the space $(\Omega\times H,\sigma(\mathcal{F}\times\mathcal{H}))$, we have, for every measurable rectangle $R_1\times R_2$, 
\begin{equation}\label{law}
\P_{p}^K(R_1\times R_2)=\displaystyle\int_{R_1}\P_p^{\omega}(R_2)\mbox{d}\P^K(\omega),\,\,\,\,\,\,\,R_1\in\mathcal{F}, R_2\in\mathcal{H}.
\end{equation}

Regarding the $K$-truncated long-range model, we say the word $(\xi_1,\xi_2,\dots)\in\Xi$ \textbf{is seen from the vertex $v\in\Z^2$} in the configuration $\omega\times\eta$ if there is a sequence $(v= v_0,v_1,v_2,\dots)$, $v_i\in\Z^2$, such that
$v_i\neq v_j$ for all $i,j\geq 0$, $i\neq j$, $e_i=\{ v_{i-1},v_i\} \in \E$, $||v_{i-1}-v_i||\leq K$, $\eta(v_i)=\xi_i$, $\omega(e_i)=1$, $i\geq 1$. 

In \cite{GLS}, the authors study a model on the $d$-dimensional lattice, $d\geq 3$, with one-dimensional long-range connections. They show that if the sum of the $p_i$ diverges, then the event $\{W_0=\Xi\}$ has probability arbitrarily close to 1, provided $K$ is sufficiently large. For $d=2$, the analogous result under the same assumption on the sum of the $p_i$ remains open. Indeed, as previously noted, it is still unresolved whether percolation occurs when relying solely on the divergence of the sum of the $p_i$. 

In this work, we address the truncation question in the context of percolation of words on a random subgraph of $\G=(\Z^2,\E_V\cup \E_H)$. We will show (see Theorem \ref{teo1}) that, for sufficiently large $K$, depending on the connection probabilities, the event $\{W_0=\Xi\}$ occurs with positive probability for a broad class of anisotropic truncated long-range percolation models for which connection probabilities $p_i$ satisfy certain regularity conditions; see Section \ref{remarkfinal}. To streamline the presentation and avoid unnecessary complications, we focus on the particular case where 
\begin{equation}\label{log}
p_i\geq (i\log i)^{-1},
\end{equation}
for all $i$ sufficiently large. This choice not only simplifies the analysis but also allows us to establish a stronger result compared to \cite{SSV}, which shows that percolation occurs under the same hypotheses.

In what follows, we consider the model in \eqref{law} with $p_{V}=\varepsilon$ and $p_{H,i}=p_i$ for all $i \geq 1$. We denote the relevant measure by $\P^K_{p,\varepsilon}$. Our main result is the following.

\begin{theorem}\label{teo1} If $p_i\geq(i\log i)^{-1}$ for all sufficiently large $i$, then for all $p\in(0,1)$ and for all $\eps\in(0,1]$, there exists a constant $K=K(p,\eps,(p_i))\in\N$ such that $$\P_{p,\varepsilon}^K(W_{0} =  \Xi )>0.$$
\end{theorem}

Consider Bernoulli site percolation on $\G$ with parameter $p\in(0,1)$. Suppose the sequence $(p_i)_{i\geq 1}$ is such that the $K$-truncated model percolates at the origin with positive probability. With coexistence of infinite clusters of 0s and 1s, it is reasonable to expect that all words can be seen, possibly with a suitably larger truncation constant. Motivated by this observation, we propose the following conjecture.

\begin{Conjecture}\label{conj} Consider a supercritical mixed long-range percolation model on $\G$ with $p_{V}=\varepsilon$ and $p_{H,i}=p_i$ for all $i \geq 1$. If the $K$-truncated process percolates for some $K\in\N$, then there exists $K_1\in\N$ such that all words can be seen from the origin with positive probability in the $K_1$-truncated model. 
\end{Conjecture}

As in Theorem~\ref{teo1}, the following result also strengthens Conjecture \ref{conj}. As noted in \cite{FLS}, the truncation question has an affirmative answer on $\Z^2$ under the condition $\limsup p_i>0$. We establish a stronger result under the same conditions on $p_i$, showing that all words can be seen from the origin with positive probability on $\G$ when $p_{V}=\varepsilon$ and $p_{H,i}=p_i$ for all $i \geq 1$. 

\begin{theorem}\label{teo2} If $\limsup{p_i}>0$, then for all $p\in(0,1)$, $\varepsilon\in(0,1]$, and $\alpha>0$, there exists $K=K(p,\varepsilon,(p_i),\alpha)\in\N$ large enough such that $$\P_{p,\varepsilon}^K(W_0=\Xi)>1-\alpha.$$
\end{theorem}

The remainder of the paper is organized as follows. Section~\ref{sec_2} contains the proof of Theorem~\ref{teo1}, which is structured into several parts: we first describe the renormalization setting and a certain exploration process on the vertices of the renormalized lattice. We then complete the proof by constructing a coupling between the process induced by the exploration and an independent, highly supercritical oriented percolation process. Sections~\ref{proof_lemma2}, \ref{proofoflemmas}, and \ref{proofofentropy} establish three key lemmas that support the arguments in Section~\ref{sec_2}. In Section~\ref{remarkfinal}, we discuss a generalization of Theorem~\ref{teo1}, and in Section~\ref{sec_3} we prove Theorem~\ref{teo2}.

\section{Proof of Theorem~\ref{teo1}}\label{sec_2}

\subsection{Proof overview}\label{proof_overview}

The proof of Theorem \ref{teo1} relies on dynamically constructing a renormalized process and coupling it with a highly supercritical oriented percolation model on $\LL^2_+$. This construction is such that the letters of a word are seen at the good vertices of the corresponding renormalized process. Following the approach in \cite{GLS}, the lattice $\LL^2_+$ is partitioned into slices that grow exponentially, with the occurrence of good events in one slice implying, with high probability, the occurrence of good events in the subsequent slice. The argument proceeds through a quantitative analysis focused on finite words, as in \cite{NTT}.

We remark that this strategy could yield a positive answer to the truncation question in percolation of words, provided that a suitable coupling can be constructed for each word. The main difficulty is that this typically requires many connections within each renormalized vertex, and hence many letters to be seen in each vertex. Without sufficient control, the number of finite words that must be considered in each slice becomes too large to obtain effective bounds.

It is worth mentioning that, in \cite{GLS}, a $d$-dimensional model with $d \geq 3$ is considered. The extra dimension allows the coupling to be built by connecting only two letters in each renormalized vertex, making the strategy feasible. Since we are working in a $2D$ setting, several technical difficulties arise, as will become clear throughout the paper.

The proof is organized into four main parts:

\noindent \textbf{Part I}. In Section \ref{renor}, we define the renormalized lattice, in which each vertex corresponds to a one-dimensional horizontal box of size $n$, and we introduce the notion of slices in this lattice. 

\noindent\textbf{Part II.} For each $\xi \in \Xi$, we build a coupling with oriented percolation on the slices of the renormalized lattice through an exploration process. To increase the probability of good events, we allow multiple explorations, which introduces many complications. We define the exploration process in Section \ref{exp_proc}.

\noindent \textbf{Part III}. For the $\ell$-th slice,  the exploration process can be restricted to the set $\{0,1\}^{s(\ell + 1)}$ of finite words of size $s(\ell + 1)$, for some function $s$ (see \eqref{sl}). Under suitable bounds for this exploration process, and by applying the union bound over $\{0,1\}^{s(\ell + 1)}$, we show that the probability that there exists $\xi \in \{0,1\}^{s(\ell + 1)}$ such that all corresponding explorations fail in the $\ell$-th slice is summable in $\ell \in \N$; see Section \ref{sec_final}.

\noindent \textbf{Part IV}. The final part of the proof establishes the estimates required in Part III. In Section \ref{proof_lemma2}, we construct the coupling with oriented percolation underlying the exploration procedure. In Section~\ref{proofoflemmas}, we build a deterministic finite configuration from which a suitable number of letters of every word can be seen starting from the origin, yielding a concatenation argument that connects the origin to the first explored slice (see Remark~\ref{weekversion}). Finally, in Section~\ref{proofofentropy}, we show that each exploration reveals only a controlled number of letters per slice (entropy control), a property crucial for making the union bounds effective.

In Section~\ref{sectionG_proc}, we define a growth algorithm within renormalized vertices for which, for sufficiently large $n$, each exploration yields a good vertex with arbitrarily high probability, independently of previous explorations and uniformly over $\xi\in\Xi$. This implies that each exploration stochastically dominates a highly supercritical oriented site percolation model, from which the required bounds follow.
  
The growth algorithm inside a box is inspired by the method of~\cite{MSV}, which was used to prove percolation, i.e., that the word $(1,1,\dots)$ is seen from the origin with positive probability. In our setting, the vertex labels ($0$ or $1$) are also inspected. At each step, we first check the edges and proceed to examine exactly $L$ vertices only if at least $L$ edges are open, where $L$ depends on $p$ and on the renormalized process. Thus, only $L$ letters are revealed per step, a key feature ensuring that the algorithm succeeds with high probability uniformly over all sequences.  Moreover, to obtain the bounds required in Part~III, multiple exploration processes may be considered, leading to a more intricate construction. By controlling the number of inspected vertices, we establish a stochastic dominance with respect to a simple random walk on~$\Z_+$ with arbitrarily large drift to the right.

\subsection{The renormalization}\label{renor}


The renormalized lattice $\mathcal{L}_n=\mathcal{L}_n(\Z^2_+)$ is defined as follows. For each $n\in\N$ and $\u=(u_1, u_2) \in \mathbb{Z}^2_+$, define the one-dimensional horizontal box 
\begin{equation}\label{renormalization}
S_{\u}^{n} = S_{(u_1, u_2)}^{n} \coloneq \{ (y_1, y_2) : y_1 \in [nu_1, n(u_1 +1)-1], y_2 = u_2\},
\end{equation}
which we call a \textbf{renormalized vertex}. Two renormalized vertices $S_{\u}^{n}$ and $S_{\v}^{n}$ are said to be adjacent if $\left \| S_{\u}^{n} - S_{\v}^{n} \right \| :=  \inf_{x\in S_{\u}^{n}, y\in S_{\v}^{n}} \left \| x - y \right \| = 1$. It follows that $\mathcal{L}_n(\Z^2_+)$ is isomorphic to $\mathbb{L}^2_+$, the nearest-neighbor square lattice restricted to the first quadrant.

We define the oriented boundary of a set $A\subset\Z^2_+$ as
$$\partial_e A:=\{\u\in A^c\colon  \exists~\v\in {A}\mbox{ such that } \u-\v=(1, 0) \mbox{ or } \u - \v = (0,1)\}.$$

With some abuse of notation, let $\u_1,\u_2,\dots$ be a fixed ordering of the vertices of $\mathcal{L}_n$. We describe an exploration process on $\mathcal{L}_n$ in which the fixed $\xi\in\Xi$ is seen. The renormalized lattice is partitioned into slices, which are explored sequentially, one at a time.

\begin{figure}[ht]
\centering
\begin{tikzpicture}[scale=0.25]
\fill[lightgray!25] (6,0) -- (20,0) -- (0,20)  -- (0,6) -- cycle;

\draw[line width=0.05cm,lightgray] (0,6) -- (6,0);
\draw[line width=0.05cm,gray] (0,0) -- (0,20);
\draw[line width=0.05cm,gray] (0,0) -- (20,0);
\draw[line width=0.02cm,dashed] (0,20) -- (5,15);
\draw[line width=0.02cm,dashed] (15,5) -- (20,0);
\draw[line width=0.05cm,black] (5,15) -- (15,5);

\draw[line width=0.02cm,dashed] (0,5) -- (1.25,3.75);
\draw[line width=0.05cm,black] (1.25,3.75) -- (3.75,1.25);
\draw[line width=0.02cm,dashed] (3.75,1.25) -- (5,0);

\draw (0,0)node[below] {$0$};
\draw (20,0)node[below] {$4^{\ell+1}-1$};
\draw (6.2,0)node[below] {$4^{\ell}$};

\draw (3,20.0)node[below] {\normalsize ${F}_{\ell +1}$};
\draw (11.5,12.0)node[below] {\normalsize $\overline{F}_{\ell +1}$};
\draw (1.5,3.0)node[below] {\normalsize $\overline{F}_{\ell}$};
\draw (7,7)node[below] {\large $T_{\ell}$};
\draw (-1,6.0)node[below] {\normalsize ${F}_{\ell}$};

\foreach \Point/\PointLabel in {(20,0)/, (6,0)/ }
        \draw[fill=black] \Point circle (0.15) node[above right] {$\PointLabel$};
\end{tikzpicture}
\caption{An illustration of the sets $T_\ell$ (gray area), $F_\ell$ (lowest diagonal line), and $\overline{F}_\ell$ (black continuous line).} 
\label{slice_schemme}
\end{figure}

For each $\ell\in \Z_+$, we define the $\ell$-th \textbf{slice} of $\mathcal{L}_n$ as
\begin{equation*}\label{slice}
T_{\ell}=\{\z = (z_1,z_2)\in\Z_+^2 : 4^{\ell}\leq z_1+z_2 < 4^{\ell+1}\}.
\end{equation*}
Also, write 
\begin{equation}\label{line}
F_{\ell}=\{\z = (z_1,z_2)\in\Z_+^2 : z_1+z_2=4^{\ell} - 1\},
\end{equation}
and let $\overline{F}_{\ell} \subset F_{\ell}$ be its middle segment, that is,
\begin{equation}\label{F2}
\overline{F}_{\ell} =\{\z = (z_1,z_2) \in F_{\ell}: 4^{\ell-1}\leq z_1 < 3\times4^{\ell-1}\};
\end{equation}
see Figure \ref{slice_schemme} for a sketch of these sets.

In our construction, each slice may be explored multiple times. In each exploration, we check whether a certain event occurs: if so, we move to the next slice; otherwise, we restart from a different set of vertices within the same slice. A key difficulty is to control the explorations so that the states of each vertex and each edge in a slice are checked at most once.

We introduce an auxiliary constant $C > 0$ such that each slice will be explored at most 
\begin{equation}\label{h}
    h = h(n) \coloneq \lceil C\log(n) \rceil
\end{equation}  
times. Later, we will require $n$ and $C$ to be large, depending on the model parameters.

\subsection{The exploration processes}\label{exp_proc}

The exploration processes are defined for each fixed word $\xi \in \Xi$, and we refer to the exploration of $\xi$ as the $\xi$-exploration. The idea is to construct paths through good renormalized vertices, guided by the letters of $\xi$. If the $\xi$-exploration succeeds in every slice of $\mathcal{L}_n$, then we will be able to show that $\xi \in W_0$.

Let us introduce some terminology. Given $\xi \in \Xi$, $K\in\N$, and $m \in \N$, a vertex $x \in \Z^2$ is said to be $(\xi,m)$-\textbf{special} in the configuration $\omega\times\eta$ if there is a sequence $(0= v_0,v_1,v_2,\dots, v_m = x)$, $v_i\in\Z^2$, such that
$v_i\neq v_j$ for all $0\leq i,j\leq m$, $i\neq j$, $e_i=\{ v_{i-1},v_i\} \in \E$, $||v_{i-1}-v_i||\leq K$, $\eta(v_i)=\xi_i$, $\omega(e_i)=1$, $i=1,\dots,m$. We say $x$ is $\xi$-\textbf{special} if $x$ is $(\xi,m)$-special for some $m \in \N$.

Under the truncated law, the underlying graph has bounded degree. Therefore, the statement that, for all $m \geq 1$, there exists a $(\xi,m)$-special vertex is equivalent to saying that $\xi$ is seen from the origin. This is also equivalent to the existence of infinitely many $\xi$-special vertices, that is,
\begin{equation*}
    \{\xi \in W_0\} = \bigcap_{m\geq 1}\Big[ \{x : x \textrm{ is $(\xi,m)$-special }\} \neq \emptyset \Big] =\Big\{\big\vert\{x : x \textrm{ is $\xi$-special}\} \big\vert = \infty \big\}.
\end{equation*}

We also introduce the increasing sequence $(s(\ell))_{\ell \in \N}$ defined by
\begin{equation}\label{sl}
    s(\ell) := 2\cdot 4^{\ell}h, 
\end{equation}
where $h$ is given in \eqref{h}. In our construction, only the first $s(\ell)$ letters of $\xi$ are relevant for the $\xi$-exploration in $T_{\ell}$ (note that both $s(\ell)$ and the number of possible explorations in $T_{\ell}$ depend on $h$). Hence, to bound the probability that there exists a word whose $h$ explorations in $T_{\ell}$ all fail, we apply a union bound with a factor $2^{s(\ell)}$, which we interpret as the entropy generated by the exploration processes.

\begin{definition}\label{seed_def} Given $\xi\in\Xi$ and $\ell\in\Z_+$, we say that a macroscopic set $\Theta_{\ell} \subset \overline{F}_{\ell}$ is a $(\xi,\ell)$-\textbf{seed} if the following conditions hold:
\begin{itemize}
    \item[S1.] $\vert \Theta_{\ell} \vert \geq 4^{\ell - 1}$;
    \item[S2.] for each $\mathbf{v}\in \Theta_{\ell}$, there exist  disjoint microscopic sets $H_{\mathbf{v}}^1, H_{\mathbf{v}}^2, \dots, H_{\mathbf{v}}^h\subset S_{\mathbf{v}}^n$, with $|H_{\mathbf{v}}^i|\geq h$, for all $i=1,\dots,h$;
    \item[S3.] for each $i=1,\dots,h$, the sequence $(t(x))_{x \in H_{\mathbf{v}}^i}$ is well defined, where each $x \in H_{\mathbf{v}}^i$ is $(\xi,t(x))$-{special} with $t(x) \leq s(\ell)$. 
     \end{itemize}
\end{definition}

Observe that each set $H_{\mathbf{v}}^i$ consists of $\xi$-special vertices at the microscopic scale within the renormalized vertex $S_{\mathbf{v}}^n$. Recall that, by definition, if a vertex $x$ is $(\xi,t(x))$-special, then the word $\xi$ is seen from the origin up to the letter $\xi_{t(x)}$ at $x$. 

We introduce a parameter $\ell_0\in\N$, which will be chosen large (see Lemma~\ref{lemmaentropy}), and start the explorations from the slice $T_{\ell_0}$. In the definition of a $(\xi,\ell_0)$-seed, we replace Condition S3 with the stronger requirement that
\begin{itemize}
    \item[S3'.] for each $i=1,\dots,h$, the sequence $(t(x))_{x \in H_{\mathbf{v}}^i}$ is well defined, where each $x \in H_{\mathbf{v}}^i$ is $(\xi,t(x))$-{special} with $t(x) \leq 2\cdot4^{\ell_0}$.
\end{itemize}

Consider the event
\begin{equation}\label{Bell0}
    \mathcal{B}_{\ell_0-1}(\xi) = \left\{\textrm{there exists a $(\xi,\ell_0)$-seed } \Theta_{\ell_0} \subset \overline{F}_{\ell_0} \right\}.
\end{equation}

Conditioned on the event $\mathcal{B}_{\ell_0-1}(\xi)$, the $\xi$-exploration is constructed inductively across the slices $T_\ell$, $\ell \geq \ell_0$. We explore $T_{\ell}$ only if the $\xi$-exploration in $T_{\ell-1}$ is successful, producing a $(\xi,\ell)$-seed $\Theta_{\ell}\subset \overline{F}_{\ell}$. Assuming that $\Theta_{\ell_0}, \dots, \Theta_{\ell}$ have been defined, we now describe the exploration within $T_{\ell}$.

Consider the sets $H_{\v}^i$, $\v \in \Theta_{\ell}$, $i=1,\dots,h$,  each consisting of $\xi$-special microscopic vertices. We begin with the first exploration starting from $H_{\v}^1$. If this exploration succeeds, we proceed to explore the slice $T_{\ell+1}$; otherwise, we initiate the second exploration in $T_{\ell}$, starting from $H_{\v}^2$, and so on.  We refer to the exploration originating from $H_{\v}^i$ as the $i$-th $\xi$-exploration, for $i = 1, \dots, h$.

Let us explain the 1st $\xi$-exploration in more detail. For each $\v\in \Theta_{\ell}$, set $G_{\mathbf{v}}^1  \coloneq H_{\mathbf{v}}^1$.  Let $(C_m^1,D_m^1)_{m\geq 0}$ be a sequence of ordered pairs of subsets of $\mathcal{L}_n$, initialized as $C_0^1=\Theta_{\ell}$ and $D_0^1=\emptyset$. After $m$ steps, the set $C_m^1$ represents the renormalized vertices through which the letters of $\xi$ are seen, and $D_m^1$ denotes the set of renormalized vertices that have already been analyzed. 

Assume $(C_r^1,D_r^1)$, $r=0,\dots,m$, has been defined and that, for each $\v\in C_m^1$, $G_{\mathbf{v}}^1$ and $(t(x))_{x \in G_{\mathbf{v}}^1}$ are well defined, with $|G_{\mathbf{v}}^1| \geq h$.  If $\partial_e C_m^1\cap (D^1_m)^c \cap T_{\ell} =\emptyset$, we stop the 1st $\xi$-exploration and declare $(C_{s}^1,D_{s}^1)=(C_m^1,D_m^1)$ for all $s\geq m$. If $\partial_e C_m^1\cap (D^1_m)^c\cap T_{\ell}\neq\emptyset$, let $\x$ be the earliest vertex according to some predefined fixed ordering in $\partial_e C_m^1\cap (D^1_m)^c \cap T_{\ell}$. 

Take $\y \in  C_m^1$ such that $\x-\y=(1, 0) \mbox{ or } \x - \y = (0,1)$. We first consider the case $\x-\y=(1, 0)$. Given $L\in\N$ (which will be chosen later),  we say $\x$ is \textbf{1-nice} if the following conditions hold: 
\begin{itemize}
\item[H1.] There are distinct vertices $x_1,\dots,x_L\in S_{\x}^n$ and (not necessarily distinct) $y(x_1),\dots,y(x_L) \in G_{\y}^1\subset S^n_{\y}$  such that $\omega(\{ y(x_k),x_k \})=1$;

\item[H2.] $\eta(x)=\xi_{t(y(x))+1}$ for some $x\in\{x_1,\dots,x_L\}$.
\end{itemize}

In the case $\x-\y=(0, 1)$,   we say $\x$ is \textbf{1-nice} if the following conditions hold: 
\begin{itemize}
\item[V1.] There are distinct vertices  $x_1,\dots,x_{L}\in S_{\x}^n 
$ and (not necessarily distinct) $y(x_1),\dots,y(x_L) \in G_{\y}^1\subset S^n_{\y}$ such that $\omega(\{ y(x_k),x_k - (0,1) \}) =1$ and $\omega(\{ x_k - (0,1) ,x_k \}) = 1$. 

\item[V2.] $\eta(x-(0,1))=\xi_{t(y(x))+1}$ and $\eta(x)=\xi_{t(y(x))+2}$, for some $x\in\{x_1,\dots,x_L\}$. 

\end{itemize}

\begin{remark} In Condition H1 (or V1), we first check the states of the edges connecting $G_{\y}^1\subset S^n_{\y}$ to $S_{\x}^n$; see Figures \ref{h1h2} and \ref{v1v2}. Then, in Condition H2 (or V2), we check the state $\eta(x)$ of only $L$ vertices in $S_{\x}^n$. This is crucial to control the size of unavailable vertices in $S_{\x}^n$ at each possible future exploration.
\end{remark}

\begin{remark} Condition~V1 requires two open edges to connect $y(x_j)$ to $x_j$. This is necessary because $G_{\y}^1$ is small compared to $n$. Thus, relying solely on edges of the form $\{x_k = y(x_k)+(0,1)\}$ with $y(x_k) \in G_{\y}^i$ does not suffice  to ensure the desired connection with high probability; see the discussion around \eqref{upsilon}.  
\end{remark}

\begin{figure}[ht]
\centering
\begin{tikzpicture}[scale = 0.85]

\draw[rounded corners, thick, dashed] (-5.8,-0.6) rectangle (2.4,0.4)
    node[midway,below,yshift=-20pt]{\textbf{$S_{\y}^n$}};

\draw[rounded corners, thick, dashed] (2.6,-0.6) rectangle (10.6,0.4)
    node[midway,below,yshift=-20pt]{\textbf{$S_{\x}^n$}};

\foreach \x in {-4.0,-3.0,...,0,2}{
  \fill (\x,0) circle (2pt);
}

\tznode(-2,0){ \scriptsize \textbf{$y(x_{3}) $}}[b];
\tznode(0,0){ \scriptsize \textbf{$y(x_{1} ) = y(x_4)$}}[b];
\tznode(2,0){ \scriptsize \textbf{$y(x_{2})$}}[b];

\foreach \x in {3,...,10}{
  \ifnum\x=5\relax
  \else
    \ifnum\x=7\relax
    \else
      \fill (\x,0) circle (2pt);
    \fi
  \fi
}

\fill (1,0) [mark=x, mark size=4pt] plot coordinates {(1,0)};
\fill (-5,0) [mark=x, mark size=4pt] plot coordinates {(-5,0)};
\tznode(3,0){ \scriptsize \textbf{$x_1$}}[b];
\tznode(4,0){ \scriptsize \textbf{$x_2$}}[b];

\fill (5,0) [mark=x, mark size=4pt] plot coordinates {(5,0)};

\fill (7,0) [mark=x, mark size=4pt] plot coordinates {(7,0)};
\tznode(6,0){ \scriptsize \textbf{$x_3$}}[b];
\tznode(9,0){ \scriptsize \textbf{$x_4$}}[b];


\tzto[ bend left = 25] (-2,0)  [above,near end] (6,0);
\tzto[ bend left = 25] (0,0)  [above,near end] (3,0);
\tzto[ bend left = 25] (2,0)  [above,near end] (4,0);
\tzto[ bend left = 25] (0,0)  [above,near end] (9,0);
\end{tikzpicture}

\caption{In Condition H1, we search within the box $S_{\mathbf{x}}^n$ for vertices $x_1, x_2, x_3, \dots, x_L$ such that each $x_i$ is connected by an open edge to some vertex in $G_{\mathbf{y}}^1\subset S^n_{\y}$. If $L$ such vertices are found in $S_{\mathbf{x}}^n$, then, in Condition H2, we check whether some vertex $x_i$, $i=1,\dots,L$, has state $\xi_{t(y(x)) + 1}$.}
\label{h1h2}
\end{figure}

For each $i = 1, \dots, h$ and $\mathbf{v} \in T_{\ell}$, let $R_{\mathbf{v}}^i$ denote the set of vertices $v \in S_{\mathbf{v}}^n$ whose states $\eta(v)$ are checked during the $i$-th $\xi$-exploration process. If Condition~H1 (or V1) does not hold, we write $R_{\mathbf{x}}^1 = \emptyset$. If Condition~H1 (or V1) holds but Condition~H2 (resp. V2) does not, we set $R_{\mathbf{x}}^1 = \{x_1, \dots, x_L \}$. 

\begin{figure}[ht]
\centering
     \begin{tikzpicture}[scale = 0.95]

\draw[rounded corners, thick, dashed] (-4.8,-0.6) rectangle (8.6,0.4)
    node[midway,below,yshift=-20pt]{\textbf{$S_{\y}^n$}};

\draw[rounded corners, thick, dashed] (-4.8,1.6) rectangle (8.6,2.6)
    node[midway,above,yshift=20pt, ]{\textbf{$S_{\x}^n$}};
\foreach \y in {-4.0,-3.0, -1, 0, 1, ...,6, 8}{
  \fill (\y,0) circle (2pt);
}  
\fill (-2,0) [mark=x, mark size=4pt] plot coordinates {(-2,0)};
\fill (7,0) [mark=x, mark size=4pt] plot coordinates {(7,0)};

\tznode(-1,0){ \scriptsize \textbf{$y(x_1)$}}[below];
\tznode(-4,0){ \scriptsize \textbf{$y(x_2)$}}[below];
\tznode(2,0){ \scriptsize \textbf{$y(x_3) = y(x_4)$}}[below];

\foreach \y in {-3.0,-2.0,...,1}{
  \fill (\y,2) circle (2pt);
}  
\foreach \y in {4,5,6,8}{
  \fill (\y,2) circle (2pt);
}  
 \fill (3,2) circle (2pt);

\tznode(-3,2){ \scriptsize \textbf{$x_1$}}[above];
\tznode(0,2){ \scriptsize \textbf{$x_2$}}[above];
\fill (2,2) [mark=x, mark size=4pt] plot coordinates {(2,2)};
\tznode(6,2){ \scriptsize \textbf{$x_4$}}[above];
\tznode(4,2){ \scriptsize \textbf{$x_3$}}[above];
\fill (-4,2) [mark=x, mark size=4pt] plot coordinates {(-4,2)};
\fill (7,2) [mark=x, mark size=4pt] plot coordinates {(7,2)};

\tzto[ bend left = 25] (2,0)  [below,near end] (6,0);
\tzto[ bend left = 25] (-4,0)  [below,near end] (0,0);
\tzto[ bend left = 25] (-3,0)  [below,near end] (-1,0);
\tzto[ bend left = 25] (2,0)  [below,near end] (4,0);
\tzto[ ] (-3,0)  [below,near end] (-3,2);
\tzto[ ] (0,0)   [below,near end] (0,2);
\tzto[ ] (4,0)   [below,near end] (4,2);
\tzto[ ] (6,0)   [below,near end] (6,2);

\end{tikzpicture}

\caption{An illustration of a vertical connection between the renormalized vertices $S^n_{\y}$ and $S^n_{\x}$.}
\label{v1v2}
\end{figure}

In Section~\ref{sectionG_proc}, we  formally define a \textbf{growth algorithm} within the box $S^n_{\u}$, $\u \in T_{\ell}$. Each execution of the algorithm examines only the vertices whose states are still unknown at that point.   
For every execution of the growth algorithm, $B_1$ will denote the initial set of unavailable vertices (those whose states were already examined). 
The algorithm starts from a vertex 
\begin{equation}\label{uzero}u_0 \in S^n_{\u}\end{equation}
such that, for some $t(u_0)$, $\xi$ is seen up to the letter $\xi_{t(u_0)}$ at $u_0$, in other words, $u_0$ is $(\xi,t(u_0))$-special. We distinguish between two cases:
\begin{itemize}
\item If $\u \notin F_{\ell+1}$, the algorithm is declared successful if it finds at least $h$ (as in \eqref{h}) $\xi$-special vertices $u \in S^n_{\u}$ for which there exists a sequence of vertices $ u_1, \dots, u_{k(u)} = u$, inside $S_{\u}^n$, satisfying $\o(\{u_{j-1}, u_{j}\}) = 1$ and $\eta(u_j) = \xi_{t(u_0)+j}$ for $j = 1, \dots, k(u)$. The set of such vertices is denoted by $A_T$, where $T$ is a stopping time for the algorithm; see \eqref{stopping}.
\item If $\u \in F_{\ell+1}$, a larger number of vertices is required to enable $h$ $\xi$-exploration processes in the next slice $T_{\ell+1}$; the algorithm is declared successful if $|A_T| \geq h^2$. 

\end{itemize}
 
When the algorithm succeeds, each $u \in A_T$ is assigned $t(u) = t(u_0)+k(u)$, with $k(u)$ defined as in the previous paragraph. By construction, each such $u$ is $(\xi,t(u))$-special. 

Returning to the definition of the 1st $\xi$-exploration process, if $\x$ is \textbf{1-nice}, we say that $\x$ is \textbf{1-good} if the following condition holds:
\begin{itemize}
\item[C3.] Let $x$ be the vertex obtained in Condition~H2 (or V2). The growth algorithm in $S_{\x}^n$, with initial set of unavailable vertices $B_1=\{x_1,\dots,x_L\}$, starting from $x$ with $t(x)=t(y(x))+1$, is successful.
\end{itemize}

Proceeding inductively, define $D_{m+1}^1 = D_m^1 \cup \{\x\}$. If $\x$ is \textbf{1-good}, set $C_{m+1}^1 = C_m^1 \cup \{\x\}$; otherwise, set  $C_{m+1}^1 = C_m^1$. In the former case, we also set $G^1_{\x} = A_T$. As observed above, the collection $(t(x))_{x \in G_{\x}^1}$ is well defined. Moreover,  $|G^1_{\x}| \geq h$ if $\x \notin F_{\ell+1}$, and $|G^1_{\x}| \geq h^2$ if $\x \in F_{\ell+1}$.

Also, after running the growth algorithm in $S_{\x}^n$,  if $\x$ is \textbf{1-nice},  the set $R_{\x}^1$ consists of all vertices checked during the growth algorithm, together with the vertices $\{x_1, \dots, x_L\}$ checked in Condition H2.

\begin{remark} Observe that when a renormalized vertex is declared good, it contains at least $h$ special vertices. This ensures that, with high probability, we can establish connections to special vertices in a neighboring box of the renormalized lattice.
\end{remark}

Define $$\Pi_{\ell}^1(\xi) = \bigcup_{m\geq 1} C_m^1,$$ the set of $1$-good vertices obtained during the 1st $\xi$-exploration in the slice $T_{\ell}$. We say that the 1st $\xi$-exploration in $T_{\ell}$ is \textbf{successful} if 
\begin{equation}\label{successful2}
    |\Gamma_{\ell}^1(\xi)| := \big| \Pi_{\ell}^1(\xi) \cap \overline{F}_{\ell+1} \big| \geq 4^{\ell}.
\end{equation}

If the 1st $\xi$-exploration in $T_{\ell}$ fails, we begin the 2nd $\xi$-exploration of vertices in $T_\ell$. More generally, for $i = 1, \dots, h-1$, if the $i$-th $\xi$-exploration in $T_{\ell}$ fails, we start the $(i+1)$-st $\xi$-exploration as we describe now.  We follow the same ordering of the vertices of $T_{\ell}$, starting the exploration from the sets $\{ G_{\v}^{i+1} := H_{\v}^{i+1} : \v\in \Theta_{\ell} \}$, recalling that for each $\v\in \Theta_{\ell}$, $t(v)$ is well defined for all $v\in G_{\mathbf{v}}^{i+1}\subset S^n_{\v}$.

The procedure is analogous to the 1st exploration. The only difference is that, for each $\x \in T_{\ell}$, vertices in the set $\cup_{j=1}^{i}R_{\x}^j$ (those checked in previous explorations) are no longer available. To simplify notation, we denote this set of unavailable vertices by
\begin{equation}\label{unchecked}
    \mathcal{U}_{\x}^{i+1} = \bigcup_{j=1}^{i}R_{\x}^j.
\end{equation}

Analogously, we inductively construct a sequence $(C_m^{i+1},D_m^{i+1})_{m\geq 0}$ of ordered pairs of subsets of $\mathcal{L}_n$,  with $(C_0^{i+1},D_0^{i+1})=(\Theta_{\ell},\emptyset)$. We only point out the differences here. 

In the definitions of $(i+1)$-nice and $(i+1)$-good vertices, conditions H1, V1 and C3, become:
\begin{itemize}
\item[H1'.] There exist distinct vertices $x_1,\dots,x_L\in S_{\x}^n\cap [\mathcal{U}_{\x}^{i+1}]^c$ and  $y(x_1),\dots,y(x_L) \in G_{\y}^{i+1}\subset S^n_{\y}$ (not necessarily distinct) such that $\omega(\{ y(x_k),x_k \})=1$;

\item[V1'.]  There exist distinct vertices  $x_1,\dots,x_{L}\in S_{\x}^n \cap (\mathcal{U}^{i+1}_{\x})^c
$ and  $y(x_1),\dots,y(x_L) \in G_{\y}^{i+1} \subset S^n_{\y}$ (not necessarily distinct) such that $\omega(\{ y(x_k),x_k - (0,1) \}) =1$ and $\omega(\{ x_k - (0,1) ,x_k \}) = 1$;

\item[C3'.] Let $x$ be the vertex obtained in Condition~H2 or V2. The growth algorithm in $S_{\x}^n$, with set of unavailable vertices $B_1=\{x_1,\dots,x_L\} \cup \mathcal{U}_{\x}^{i+1}$, starting from $x$ with $t(x)=t(y(x))+1$, is successful.
\end{itemize}

Since the exploration is restricted to $S_{\x}^n\cap[\mathcal{U}_{\x}^{i+1}]^c$, each vertex is checked at most once. Moreover, the sets of special vertices $G_{\x}^j$, $j=1,\dots,h$, are disjoint, ensuring that when verifying Condition H1' - which involves edges from $G_{\y}^i$ to $S_{\x}^n\cap\mathcal{U}_{\x}^{i+1}$ - no edge is checked more than once. In contrast, Condition V1' requires checking edges between vertices of $G_{\y}^i$ and other vertices in $S_{\y}^n$, which may have already been examined during the construction of $G_{\y}^i$. To avoid this, we assume without loss of generality (by monotonicity of $[i\log(i)]^{-1}$) that the growth algorithm uses only edges of odd length, while Condition V1' involves only even-length edges. Finally, since the growth algorithm in Condition C3' does not consider the set of unavailable vertices $B_1$, it follows — together with the last item of Remark~\ref{remark_1} — that each edge inside a renormalized box is checked in at most one execution of the growth algorithm.

If $\x$ is $(i+1)${-good}, we set $G^{(i+1)}_{\x} = A_T$ and note that $(t(x))_{x \in G_{\x}^{i+1}}$ is well defined. Moreover, either $|G^{i+1}_{\x}| \geq h$ if $\x \notin F_{\ell+1}$, or $|G^{i+1}_{\x}| \geq h^2$ if $\x \in F_{\ell+1}$.

Similarly, we define $$\Pi_{\ell}^{i+1}(\xi) = \bigcup_{m\geq 1} C_m^{i+1},$$ the set of ${(i+1)}$-good vertices. We say that the ${(i+1)}$-st $\xi$-exploration in $T_{\ell}$ is successful if 
\begin{equation}\label{successful3}
    |\Gamma_{\ell}^{i+1}(\xi)| := \big| \Pi_{\ell}^{i+1}(\xi) \cap \overline{F}_{\ell+1} \big| \geq 4^{\ell}.
\end{equation}


Having defined the renormalization and all the $h$ potential $\xi$-explorations in the slice $T_{\ell}$, write 
\begin{equation}\label{B1}
    \mathcal{B}_{\ell}^i(\xi) = \left\{|\Gamma_{\ell}^i(\xi)| \geq 4^{\ell}\right\}
\end{equation}
for the event where the $i$-th $\xi$-exploration is successful, and let 
\begin{equation}\label{B2}
\mathcal{B}_{\ell}(\xi) = \bigcup_{i=1}^h \mathcal{B}_{\ell}^i(\xi)
\end{equation} 
denote the event that at least one of the $h$ explorations is successful. Figure~\ref{eventobom} illustrates an exploration in $T_{\ell}$.

\begin{figure}[ht]
\centering
\begin{tikzpicture}[scale=0.25]
\fill[lightgray!25] (6,0) -- (20,0) -- (0,20)  -- (0,6) -- cycle;

\draw[line width=0.05cm,gray] (0,6) -- (6,0);
\draw[line width=0.05cm,gray] (0,0) -- (0,20);
\draw[line width=0.05cm,gray] (0,0) -- (20,0);
\draw[line width=0.02cm,dashed] (0,20) -- (5,15);
\draw[line width=0.02cm,dashed] (15,5) -- (20,0);
\draw[line width=0.05cm,darkgray] (5,15) -- (15,5);

\draw[line width=0.02cm,dashed] (0,5) -- (1.25,3.75);
\draw[line width=0.05cm,darkgray] (1.25,3.75) -- (3.75,1.25);
\draw[line width=0.02cm,dashed] (3.75,1.25) -- (5,0);

\draw (0,0)node[below] {$0$};
\draw (20,0)node[below] {$4^{\ell+1}-1$};
\draw (6.2,0)node[below] {$4^{\ell}$};

\draw (11.5,12.0)node[below] {\normalsize $\overline{F}_{\ell +1}$};
\draw (1.5,3.0)node[below] {\normalsize $\overline{F}_{\ell}$};

\foreach \Point/\PointLabel in {(20,0)/, (6,0)/ }
        \draw[fill=black] \Point circle (0.15) node[above right] {$\PointLabel$};

        \draw (2,3) -- (3,3) --(3,5);
\draw (6,5) -- (6,9) -- (11,9);
\draw (4,1) -- (4,2);
\draw (7,3) -- (7,6);
\draw (7,7) -- (7,8)-- (10,8);
\draw(10,9) (11,9);

\draw (2,3) -- (4,3) -- (4,7);
\draw (6,7) -- (13,7);
\draw (6,5) -- (6,14);
\draw (8,9) -- (8,12);
\draw (4,5) -- (6,5);
\draw (3,3) -- (10,3);
\draw (4,7) -- (4,16);
\draw (4,13) -- (5,13) -- (5,15);
\draw (2,3) -- (2,18);
\draw (4,1) -- (10,1) --(10,2) -- (18,2);
\draw (12,2) -- (12,3) -- (17,3);


\end{tikzpicture}
\caption{An illustration of an exploration in $T_{\ell}$ (gray area). The exploration is successful if the number of vertices in $\overline{F}_{\ell+1}$ that can be reached from some vertex in $\Theta_{\ell} \subset \overline{F}_{\ell}$ is at least $4^{\ell}$.} 
\label{eventobom}
\end{figure}

The $\xi$-explorations in $T_{\ell}$ are not independent; however, the probability of success is uniformly bounded (see Lemma~\ref{lemmaboundchances}). When at least one of the explorations is successful, a large collection of paths consisting of $\xi$-good renormalized vertices reaches $\overline{F}_{\ell+1}$. Conditioning on this event, we can proceed to explore the next slice.

Indeed, if for some $i = 1, \dots, h$, the $i$-th $\xi$-exploration in $T_{\ell}$ succeeds, we will be able to perform $h$ more attempts to obtain a successful $\xi$-exploration in the next slice $T_{\ell+1}$. By definition, the $i$-th $\xi$-exploration reaches at least $4^{\ell}$ renormalized vertices $\mathbf{z}\in \overline{F}_{\ell+1}$ (see \eqref{successful3}), each satisfying  $|G_{\mathbf{z}}^{i}|\geq h^2$. If the event $\mathcal{B}^i_{\ell}(\xi)$ occurs, we set
\begin{equation*}
  \Theta_{\ell + 1}  =  \Gamma^i_{\ell}(\xi).
\end{equation*}

For each $\mathbf{z} \in \Gamma^i_{\ell}(\xi)$, the set $G_{\mathbf{z}}^{i}$ is partitioned into sets $H_{\mathbf{z}}^{1},\dots,H_{\mathbf{z}}^{h}\subset G_{\mathbf{z}}^{i}$ with $|H_{\mathbf{z}}^{j}|>h$ for all $j=1,\dots,h$, from which we resume the $\xi$-exploration in $T_{\ell+1}$. As argued above, $\Gamma^i_{\ell}(\xi)$ satisfies conditions S1 and S2. To ensure that $\Gamma^i_{\ell}(\xi)$ is indeed a $(\xi,\ell+1)$-seed, one more property must be verified. 

Recall the sequence $(s(\ell))_{\ell \in \mathbb{N}}$  in \eqref{sl} and Condition S3 in the definition of a $(\xi,\ell)$-seed. Also, recall the auxiliary constant $C$ (see \eqref{h}). We have the following lemma, whose proof is given in Section~\ref{proofofentropy}.


\begin{lemma}\label{lemmaentropy}
    Let $C$ be as in \eqref{h}. Then,  for every sufficiently large $n$, there exists $\ell_0(C,n)$ such that if the $i$-th $\xi$-exploration in $T_{\ell}$ is executed, $i=1,\dots,h$, then  $(t(x))_{x \in G_{\z}^i}$ satisfies $$t(x)\leq s(\ell+1),$$ 
    for every $\xi \in \Xi$, $\ell \geq \ell_0$, and for all $\z \in \Gamma^i_{\ell}(\xi)$.
\end{lemma}

\subsection{Proof conclusion}\label{sec_final}
In this section, we present the final part of the proof of Theorem~\ref{teo1}. Throughout, we assume the validity of two lemmas whose proofs will be given in the subsequent sections.  

Recall the events defined in \eqref{Bell0}, \eqref{B1}, and \eqref{B2}. The following lemma states that, with appropriate choices of $L$, $C$, and $n$, the probability that the $i$-th $\xi$-exploration fails is uniformly bounded. 

\begin{lemma}\label{lemmaboundchances}
    Assume the hypothesis of Theorem~\ref{teo1}. For every $a > 0$, the constants $L$ and $C(L)$  can be chosen such that, for every sufficiently large $n$, the following holds. For every $\xi \in \Xi$ and $\ell_0 \in \N$, the $\xi$-explorations in the slices $T_{\ell_0}, T_{\ell_0+1}, \dots$ satisfy
    \begin{equation}\label{conclusion2}
    \P_{p,\eps}^{2n}\left( \left(\mathcal{B}_{\ell}^i(\xi)\right)^c \Bigg\vert \big[ \bigcap_{j< i} (\mathcal{B}_{\ell}^j(\xi))^c \big] \bigcap \mathcal{B}_{\ell-1}(\xi) \right) \leq a^{4^{\ell}}, 
\end{equation}
for all $i = 1, \dots, h$ and $\ell \geq \ell_0$.
\end{lemma}

In \eqref{conclusion2}, we adopt the convention that $[\cap_{j<1} (\mathcal{B}_{\ell}^j(\xi))^c ] \cap \mathcal{B}_{\ell-1}(\xi) = \mathcal{B}_{\ell-1}(\xi) $ .

Lemma~\ref{lemmaentropy} implies that the events $\mathcal{B}_{\ell}^i(\xi)$ and $\mathcal{B}_{\ell}(\xi)$ depend only on the first $s(\ell+1)$ letters of $\xi$. Motivated by this, we define 
\begin{equation}\label{Xifinite}
    \Xi_{\ell}=\{0,1\}^{s(\ell+1)},
\end{equation} the set of finite words $(\xi_1,\xi_2,\dots,\xi_{s(\ell + 1)})$ of length $s(\ell+1)$. We also introduce the projection map
\begin{equation}\label{projeciton}
    \sigma_{\ell}:\Xi \longrightarrow \Xi_{\ell},\,\,\,\,\,\,\sigma_{\ell}(\xi_1,\xi_2,\dots)=(\xi_1,\dots,\xi_{s(\ell+1)}).
\end{equation}

Let $\xi\in\Xi_{\ell}$ and $\phi_{\xi} \in\Xi$ such that $\sigma_{\ell}(\phi_{\xi})=\xi$. For each $\ell \geq \ell_0$ and $\xi\in\Xi_{\ell}$, we define the events 
\begin{equation}\label{Bproj}
    \mathcal{B}_{\ell}^i(\xi) = \mathcal{B}_{\ell}^i(\phi_\xi),\,\,\,\,\,\, i = 1, \dots, h,
\end{equation} 
\begin{equation}\label{Bproj2}\mathcal{B}_{\ell}(\xi) = \mathcal{B}_{\ell}(\phi_\xi).
\end{equation}  
Also, write
$$\mathcal{D}_{\ell}=\bigcap_{\xi\in \Xi_{\ell}}\mathcal{B}_{\ell}(\xi).$$
For completeness, we also define (see~\eqref{Bell0}) $$\mathcal{D}_{\ell_0 - 1} = \bigcap_{\xi \in \Xi} \mathcal{B}_{\ell_0 -1}(\xi).$$ 

With this notation, $\mathcal{D}_{\ell}$ is the event that a certain random number of letters (bounded by $s(\ell + 1)$) of every $\xi\in\Xi_{{\ell}}$ is seen from the origin, and the corresponding renormalized processes reach the set $\overline{F}_{\ell+1}$ with at least half of its vertices. 

The occurrence of the event $\mathcal{D}_{\ell_0 -1 }$ is necessary for the $\xi$-exploration in the renormalized lattice $\mathcal{L}_n$ to start, for any $\xi \in \Xi$. 
The following lemma guarantees that $\mathcal{D}_{\ell_0 -1 }$ occurs with positive probability.

\begin{lemma}\label{lemmaverdaorigem}
    Assume $p_i > 0$ for all $i \geq m$, for some $m>0$. Then, for sufficiently large $n > n(C)$, 
    \[\P_{p,\varepsilon}^{4^{\ell_0} n}(\mathcal{D}_{\ell_0-1}) > 0, \textrm{ for  every } \ell_0 \in \N.  \]
\end{lemma}

Observe that the event $\mathcal{D}_{\ell_0 - 1}$ does not depend on the constant $L$. Indeed, since the explorations begin in the slice $T_{\ell_0}$, the occurrence of $\mathcal{D}_{\ell_0 - 1}$ is independent of the exploration processes. To prove Lemma~\ref{lemmaverdaorigem}, it suffices to show the existence of a configuration containing a $(\xi,\ell_0)$-seed for every $\xi \in \Xi$; such a configuration is constructed in Section~\ref{proofoflemmas}.
We are ready to prove Theorem \ref{teo1}.

\begin{proof}[Proof of Theorem \ref{teo1}] By monotonicity, it suffices to prove the theorem with $p_i = (i \log i)^{-1}$ for $i$ sufficiently large. Given $a > 0$, fix $L$ and $C$ as in Lemma~\ref{lemmaboundchances}. Pick $n$ sufficiently large (as a function of $C$ and $L$) so that Lemmas~\ref{lemmaentropy}, \ref{lemmaboundchances}, and \ref{lemmaverdaorigem} hold. Set $\ell_0 = \ell_0(C,n)$ as in Lemma~\ref{lemmaentropy}.

Recall that $W_0$ denotes the set of words seen from the origin; see \eqref{set_words}.  Under $\mathcal{D}_{\ell_0 - 1}$, if $W_{0}\neq \{0,1\}^{\N}$, then there exist $\ell \geq \ell_0$ and $\xi^*\in\{0,1\}^{s(\ell + 1)}$ such that all $h$ corresponding $\phi^*$-explorations (where $\sigma_{\ell}(\phi^{\ast}) = \xi^{\ast}$) in $T_{\ell}$ fail, that is,
\begin{equation*}
    \bigcap_{\ell \geq \ell_0 - 1} \mathcal{D}_{\ell} \subset  \big\{W_0 = \Xi \big\}. 
\end{equation*}
Therefore, taking the truncation constant $K = 4^{\ell_0}n > 2n$, we obtain
\begin{equation*}
    \P^{K}_{p,\varepsilon}\big( W_0 = \Xi \big) \geq  \left[ \prod_{\ell \geq \ell_0} \P^{2n}_{p,\varepsilon}\big( \mathcal{D}_{\ell} \big\vert \mathcal{D}_{\ell - 1} \big) \right] \P_{p,\varepsilon}^{4^{\ell_0}n}(\mathcal{D}_{\ell_0 - 1}) 
    .
\end{equation*}
By Lemma~\ref{lemmaverdaorigem}, it suffices to show that $\prod_{\ell \geq \ell_0}\P^{2n}_{p,\varepsilon}\big( \mathcal{D}_{\ell} \big\vert \mathcal{D}_{\ell - 1} \big)$ is positive, which is equivalent to showing that
\begin{equation*}
    \sum_{\ell \geq \ell_0} \left[ 1 - \P^{2n}_{p,\varepsilon}\big( \mathcal{D}_{\ell} \big\vert \mathcal{D}_{\ell - 1} \big) \right] =  \sum_{\ell \geq \ell_0} \P^{2n}_{p,\varepsilon}\big( \mathcal{D}_{\ell}^c \big\vert \mathcal{D}_{\ell - 1} \big)  < \infty
\end{equation*}
and $\P^{2n}_{p,\varepsilon}\big( \mathcal{D}_{\ell}^c \big\vert \mathcal{D}_{\ell - 1} \big) < 1$, for all $\ell \geq \ell_0$. 

We have 
\begin{align*}
\nonumber \sum_{\ell \geq \ell_0} \P^{2n}_{p,\eps}\big( \mathcal{D}_{\ell}^c \big\vert \mathcal{D}_{\ell - 1} \big)  &= \sum_{\ell \geq \ell_0}\P_{p,\eps}^{2n} \left(\bigcup_{\xi\in\Xi_{\ell}} \mathcal{B}_\ell^c(\xi) \bigg\vert\mathcal{D}_{\ell - 1}  \right) 
\\
\nonumber &\leq\sum_{\ell \geq \ell_0}\sum_{\xi\in\Xi_{\ell}} \P_{p,\eps}^{2n} \left(\bigcap_{i=1}^{h} \left(\mathcal{B}_\ell^i(\xi) \right)^c \bigg\vert  \mathcal{B}_{\ell-1}(\phi_\xi) \right)
\\
 &\leq \sum_{\ell \geq \ell_0} \sum_{\xi\in\Xi_{\ell}} \prod_{i=1}^{h} \P_{p,\eps}^{2n}\left( \left(\mathcal{B}_{\ell}^i(\xi)\right)^c \Bigg\vert \big[ \bigcap_{j< i} (\mathcal{B}_{\ell}^j(\xi))^c \big] \bigcap \mathcal{B}_{\ell-1}(\phi_\xi)  \right).
\end{align*}

Recalling $s(\ell+1) = 2\cdot 4^{\ell+1}h$ as defined in \eqref{sl}, Lemma~\ref{lemmaboundchances} yields 
\begin{align}\label{final}
     \sum_{\ell \geq \ell_0} \P^{2n}_{p,\eps}\big( \mathcal{D}_{\ell}^c \big\vert \mathcal{D}_{\ell - 1} \big)  & \leq \sum_{\ell \geq \ell_0} \sum_{\xi\in\Xi_{\ell}} a^{4^{\ell}h} 
    = \sum_{\ell \geq \ell_0} 2^{s(\ell+1)}a^{4^{\ell}h} \nonumber\\  &= \sum_{\ell \geq \ell_0} 2^{2\cdot4^{\ell+1}h}a^{4^{\ell}h}= \sum_{\ell \geq \ell_0} 4^{4^{\ell + 1}h} a^{4^{\ell}h} = \sum_{\ell \geq \ell_0}(4^4a)^{4^\ell h}.
\end{align}

Choosing $a < 4^{-4}$, we obtain $\sum_{\ell \geq \ell_0} \P^{2n}_{p,\eps}\big( \mathcal{D}_{\ell}^c \big\vert \mathcal{D}_{\ell - 1} \big)<\infty$. 
Moreover, for all $\ell\geq \ell_0$, we have 
\begin{equation}\label{menor1}\P^{2n}_{p,\varepsilon}\big( \mathcal{D}_{\ell}^c \big\vert \mathcal{D}_{\ell - 1} \big) \leq (4^4a)^{4^\ell h}  < 1.
\end{equation}
\end{proof}

\begin{remark}\label{weekversion}
    Without invoking Lemma~\ref{lemmaverdaorigem}, we obtain a weaker version of Theorem~\ref{teo1}: all words are seen from $\overline{F}_{\ell_0}$ with positive probability. Specifically, choosing $C$, $L$, $n$, and $\ell_0$ as in the proof, we have  
    $$\P_{p , \varepsilon}^{2n}\left(\bigcup_{v\in \overline{F}_{\ell_0}} W_v = \Xi\right) > 0.$$ 
     Indeed, setting $t(x) = 0$ for every $x \in S_{\x}^n$ with $\x \in \overline{F}_{\ell_0}$, we see that $\overline{F}_{\ell_0}$ forms a $(\xi, \ell_0)$-seed for all $\xi \in \Xi$. The conclusion then follows from \eqref{final} and \eqref{menor1}.
\end{remark}

\section{Proof of Lemma \ref{lemmaboundchances}}\label{proof_lemma2}
\subsection{The growth algorithm}\label{sectionG_proc}
Recall the growth algorithm introduced in Section \ref{proof_overview}. As introduced in Section \ref{renor}, the vertices of the renormalized lattice are one-dimensional boxes of size $n$ 
(see \eqref{renormalization}). Let $n\in\N$ and define $[n]=\{1,\dots,n\}$. Consider a random graph with vertex set $[n]$, where each edge $\{ i,j\}$ is independently open with probability $p_{|j-i|}$, and each vertex is independently assigned state $1$ with probability $p$ and state $0$ with probability $1-p$. This random graph can be seen as a configuration restricted to the box $S_{\x}^n$ of a renormalized vertex $\x$, where the algorithm will be executed. We denote its law by $\P_p^n$.

The parameters of the algorithm are the natural numbers $L$, $C$, $n$, and $h$, as in \eqref{h}. Given $\xi \in \Xi$, we define the growth algorithm in $[n]$ with respect to $\xi$, starting with a set of unavailable vertices $B_1$ and an initial vertex $v_1$ with $t(v_1) \in \N$; $v_1$ is identified with $u_0$ in \eqref{uzero}. Write $A_1=\{v_1\}\subset [n]$, $B_1\subset [n]$, and $Z_1=1$. For an integer $k\geq 1$, assume that we have defined $(A_k,B_k,Z_k)$ and $t(v)$ for all $v\in A_k$, where $A_k = \{v_1, \dots, v_{|A_k|}\}$ with
\begin{equation}\label{A}
    |A_k| = k + Z_k - 1.
\end{equation}
If $|A_k|\geq h$ or $Z_k=0$, we stop the algorithm. Otherwise, we define $(A_{k+1},B_{k+1},Z_{k+1})$ in two steps. Note that \eqref{A} implies $|A_k| \geq k$ whenever $Z_k \geq 1$, which ensures that $v_k \in A_k$ is well defined. 

We say that Step 1 is successful if there are distinct $u_1,u_2,\dots,u_L\in [n]\cap B_k^c$ such that $\{ v_k,u_i \}$ is open for all $i=1,\dots, L$. If Step 1 fails, we set $Z_{k+1}=Z_k-1$, $A_{k+1}=A_k$, $B_{k+1}=B_k$, and $t(v)$ is unaltered for every $v\in A_k$. If Step 1 succeeds, we set $B_{k+1}=B_k\cup\{u_1,\dots,u_L\}$ and proceed to Step 2.

We say that Step 2 is successful if there are distinct $y_1,y_2\in\{u_1,\dots,u_L\}$ such that $\eta(y_1)=\eta(y_2)=\xi_{t(v_k)+1}$; see Figure \ref{Fgrown}. The choice of two vertices $y_1$ and $y_2$ ensures that $(Z_k)_k$ is a process with positive drift; see comment after Remark \ref{remark_1}. In this case, we write $A_{k+1}=A_k\cup\{v_{|A_k|+1},v_{|A_k|+2}\}$, with $v_{|A_k|+1}=y_1$, $v_{|A_k|+2}=y_2$, $t(v_{|A_k|+1})=t(v_{|A_k|+2})=t(v_k)+1$, and $Z_{k+1}=Z_k+1$. If Step 2 fails, we write $Z_{k+1}=Z_k-1$, $A_{k+1}=A_k$, and $t(v)$ is unaltered for every $v\in A_k$. Note that \eqref{A} still holds for $k+1$ in either case. See Figure \ref{Fgrown} for an illustration of the growth algorithm.

\begin{figure}[ht]
    \centering
       \begin{tikzpicture}[scale = 0.85]
    \draw[rounded corners, thick, dashed] 
        (-7.6,-0.6) rectangle (7.7,0.6)
        node[midway,below,yshift=-23pt]{\textbf{$S_{\x}^n$}};
    \fill (-4,0) circle (2pt);
     \tznode(-6,0){ \textbf{$v_1$}}[b]
     \tzto[ bend left = 30] (-6,0)  [below,near end] (-4.0,0);
     \tzto[ bend left = 30] (-6,0)  [below,near end] (0,0);
      \fill (-2.0,0) circle (2pt);
     \tznode(-4.0,0){\textbf{$\xi_{t(v_1) +1}$}}[b]
      \tzto[ bend left = 30] (-4.0,0)  [below,near end] (-2,0);
         \tzto[ bend left = 30] (-4.0,0)  [below,near end] (2,0);
      \fill (0,0) circle (2pt);
     \tznode(-2,0){\textbf{$\xi_{t(v_1) +2}$}}[b]
     \fill (1,0) circle (2pt); 
     \tznode(0,0){\textbf{$\xi_{t(v_1) +1}$}}[b]
     \fill (2,0) circle (2pt);
     \tznode(2,0){\textbf{$\xi_{t(v_1)+2}$}}[b]
        \tzto[ bend left = 30] (0,0)  [below,near end] (4,0);     
        \tzto[ bend left = 30] (0,0)  [below,near end] (7,0);  
     \fill (4,0) circle (2pt);
     \tznode(4,0){\textbf{$\xi_{t(v_1) +2}$}}[b]
     \tznode(7,0){\textbf{$\xi_{t(v_1) +2}$}}[b]

        \foreach \x in {-7,-6,...,7}{
  \ifnum\x=3\relax
  \else
    \ifnum\x=6\relax
    \else
      \fill (\x,0) circle (2pt);
    \fi
  \fi
}
\fill (3,0) [mark=x, mark size=4pt] plot coordinates {(3,0)};

\fill (6,0) [mark=x, mark size=4pt] plot coordinates {(6,0)};

\end{tikzpicture}
    \caption{The growth algorithm in $S^n_{\x}$.}
    \label{Fgrown}
\end{figure}

Let $T'$ and $T''$ be the stopping times 
\begin{equation}\label{stopping}
T':=  \inf\{k:|A_k|\geq h\mbox{ or }Z_k=0\} \quad \textrm{and} \quad T'' :=
\inf\{k:|A_k|\geq h^2 \mbox{ or }Z_k=0\}. 
\end{equation}
When the algorithm is executed inside a box of the form $S_{\u}^n$ with $\u \notin F_{\ell}$, we say it is \textbf{successful} if $Z_{T'}>0$. When it is executed inside a box $S_{\u}^n$ with $\u \in F_{\ell}$,  we say it is \textbf{successful} if $Z_{T''}>0$. With a slight abuse of notation, we denote by $T$ either $T'$ or $T''$, depending on the respective case.

\begin{remark}\label{remark_1}
We have the following observations.
\begin{enumerate}
\item When the algorithm is run inside a box $S_{\u}^n$ starting from $v_1 \in S_{\u}^n$, with $v_1$ being $(\xi,t(v_1))$-special, then  $v_k$ is $(\xi,t(v_k))$-special for all $k = 1, \dots, T$. 

\item The algorithm is such that $t(v_k) \leq t(v_1) + k \leq t(v_1) + T$, for all $k = 1, \dots, T.$ A straightforward upper bound is $T'\leq h$ and $T''\leq h^2$. This observation will be useful in the proof of Lemma~\ref{lemmaentropy}.

\item Recall the definition of the set $B_{k+1}$ in Step 1 of the algorithm. As discussed in the previous section, the growth algorithm can potentially be executed up to $h$ times within each box. Therefore, control of the number of unavailable vertices in each execution is required. During each run of the algorithm, at most $L$ vertices are checked at each step. Hence, at the stopping time $T$, the number of unavailable vertices is bounded by
\begin{equation*}\label{bound}
|B_T|\leq |B_1|+ L(T-1). 
\end{equation*}

\item  When executing the algorithm at a vertex $v_k$, only edges incident to $v_k$ are examined. Since the vertices $v_k$ are distinct, each edge is checked at most once. Moreover, as $v_k$ becomes unavailable for subsequent explorations, each edge is checked in at most one potential execution of the algorithm.
\end{enumerate}
\end{remark}

\begin{lemma}\label{successful}
    Consider the growth algorithm in $[n]$ with  $(p_i)_{i \geq 1}$ as in Theorem~\ref{teo1}. Given $\delta > 0$ and $p \in (0,1)$,  the constant $L$ can be chosen such that, for every sufficiently large $n$, the following holds. For every $\xi \in \Xi$ and  every $B_1 \subset [n]$, with $\vert B_1 \vert \leq 3Lh^3$, 
 \begin{equation*}
     \P_p^n\big(Z_T > 0\big) \geq 1 - \delta.
 \end{equation*}
\end{lemma}
\begin{proof}  We begin with a brief overview of the argument. Taking \(L\) large increases the probability of success in Step 2, while taking \(n\) large improves the probability of success in Step 1. Thus, by choosing \(L\) and \(n\) sufficiently large, the marginals of \((Z_k)_{k=1}^T\) stochastically dominate a random walk on \(\mathbb Z_+\) whose probability of jumping to the right is arbitrarily close to one.

Once Step 1 succeeds, Step 2 holds with arbitrarily high probability (uniformly in $\xi$ and $t(v_k)$) by choosing $L$ sufficiently large, depending on $p$. Indeed, for large $L$, the probability that a binomial random variable $\textrm{Bin}(L;p^{\ast})$ exceeds $1$ is arbitrarily close to $1$, where $p^{\ast} = \min\{p, 1-p\}$.

We evaluate the growth algorithm at time $k$. We aim to show that, for any $\rho > 0$, there exist $L(\rho)$ and $n(C,L,\rho)$ sufficiently large such that
\begin{equation*}\label{marginals}
    \P_p^{n}(Z_{k+1}=j+1  \vert  Z_k=j, k < T )> 1 - \rho,
\end{equation*}
whenever $\P_p^{n}( Z_k=j, k < T  ) > 0$. 

On the event $\{ Z_k = j , k < T \}$, we can write $A_k = \{v_1, \dots, v_{k+j-1}\}$. Since $k < T \leq h^2$, we have $$|B_k| \leq |B_1| + (k-1)L < 4Lh^3.$$ Let $B_k^c$ be partitioned into $L$ disjoint sets $\Phi_1, \dots, \Phi_L$. For each $i = 1, \dots, L$, define $\Lambda_{k,i} = \{z \in [n] \cap \Phi_i : \{ v_k , z \} \textrm{ is open}\}$. Then, it holds that 
\begin{align*}
    \P_p^{n}\left( |\Lambda_{k,i}| \geq 1 \right) = \left[ 1-\prod_{z \in \Phi_i} \big( 1-p_{|v_k -z|} \big)\right]\nonumber
            \geq 1-\exp\left({- \sum_{ z \in \Phi_i} p_{|v_k -z|}}\right).\nonumber
\end{align*}

A lower bound for the probability of success in Step 1 is given by 
$$
     \prod_{i=1}^L \P_p^{n}\left( |\Lambda_{k,i}| \geq 1\right). 
$$
Now, for every sufficiently large $n$, we obtain $$\P_p^{n}\left( |\Lambda_{k,i}| \geq 1\right) \geq (1-\rho)^{1/L},\,\,\,\,\,\,\,\,\,  i = 1, \dots, L.$$ 

Indeed, on at least one side of $v_k$ (left or right), there are at least $n/2 - |B_k|$ vertices in $B_k^c$. Using the fact that $|B_k| < 4Lh^3$, the monotonicity of $(i \log i)^{-1}$, and recalling \eqref{h}, we obtain
\begin{equation}\label{divergence}
    \sum_{ z \in B^c_k} p_{|v_k - z|} \geq  \sum_{j=4Lh^3}^{n/2}(j\log j)^{-1}\xrightarrow[n \to \infty]{} \infty.
\end{equation}
Hence, for sufficiently large $n$, the set $B_k^c$ can be partitioned so that $\sum_{ z \in \Phi_i} p_{|v_k-z|}$ is arbitrarily large, for $i = 1, \dots, L$.

It follows that, taking $L$ large enough, for every sufficiently large $n$, the marginals of $(Z_k)_{k=1}^T$ stochastically dominate a simple random walk on $\Z_+$ whose probability of jumping to the right is larger than $1-\rho$. Hence, for every $\delta>0$, one can choose $\rho(\delta)$ sufficiently small so that the probability that this random walk never visits the origin exceeds $1-\delta$. Therefore, the lemma holds for all sufficiently large $n$. 

\end{proof}

\subsection{Stochastic dominance by oriented site percolation}\label{SecLemmas}

In this section, we prove Lemma~\ref{lemmaboundchances}. We begin by discussing a useful result for Bernoulli oriented site percolation on $\LL^2$. In this model, each vertex is independently open with probability $\gamma$ and closed with probability $1-\gamma$, $\gamma \in [0,1]$. Denote by $P_{\gamma}$ the corresponding law. 

For $S\subset \Z_+^2$, define
$$C_S\coloneq\{v\in \Z_+^2: S\rightarrow v\},$$ where $S\rightarrow v$ means that there exists an open oriented path from a vertex $u\in S$ to $v$.

For $\ell\in\Z_+$, recall the definitions of $F_{\ell}$ and $\overline{F}_{\ell}$ in \eqref{line} and \eqref{F2}. The following proposition is reminiscent of \cite{GLS} (see also \cite{NTT}). We omit its proof, as an equivalent statement appears in Lemma~1 of \cite{GLS}. 

\begin{proposition}\label{aux2}
Consider oriented Bernoulli site percolation on $\LL^2$ with parameter $\gamma$. For every $a > 0$, there exists $ \gamma(a) < 1$ such that for every $\gamma \geq \gamma(a)$ and every $\ell\in\Z_+$,
\begin{equation*}
    P_{\gamma}\left( |C_S \cap \overline{F}_{\ell+1}| < 4^{\ell}  \right) \leq a^{4^{\ell}},
\end{equation*}
for all $S \subset \overline{F}_{\ell}$ satisfying $|S| \geq 4^{\ell - 1}$.
\end{proposition}

In the oriented percolation process induced by the explorations on the renormalized lattice, we choose constants $C$ and $L$ such that, for sufficiently large $n$, the probability that a vertex is declared $i$-good is uniformly large, for every $\xi \in \Xi,$ independently of all steps in the current and previous explorations. Proposition~\ref{aux2} then ensures that, conditional on the event $\mathcal{B}_{\ell}(\xi)$ - which guarantees that at least half of the vertices in $\overline{F}_{\ell}$ are $i$-good for a fixed word $\xi$ and some $i$ whose $i$-th exploration succeeds in $T_{\ell-1}$ - it is highly unlikely that the fraction of good vertices in $\overline{F}_{\ell+1}$ drops below $1/2$.

Suppose that, at some point during the $i$-th $\xi$-exploration of vertices in $T_{\ell}$, we reach the renormalized vertex $\mathbf{x}$. Let $\mathcal{F}(\mathbf{x},i)$ denote the history of the process up to this step, that is, the $\sigma$-field containing all information about the renormalized vertices inspected so far, as well as the inspected bonds and vertices of the original graph, including all $(i-1)$ previous $\xi$-exploration processes.

\begin{lemma}\label{good_vert} Assume the hypothesis of Theorem~\ref{teo1}. Then, for all $p\in(0,1)$, $\eps>0$, and $\delta>0$, there exist constants $L(p,\eps,\delta)$ and $C(L)$ such that for every sufficiently large $n(L,C)$, 
\begin{equation*}\label{result}
\P_{p,\eps}^{2n}(\x\mbox{ is $i$-good } \big| \mathcal{F}(\x,i))>1-\delta. 
\end{equation*}
\end{lemma}

\begin{proof} A vertex $\x$ is declared $i$-good if it satisfies conditions H1', H2 (or V1', V2) and C3' from Section~\ref{renor}. To prove the lemma, we show that these three conditions are satisfied with arbitrarily large probability, independently of the past.

We begin by deriving an upper bound for the number of unavailable vertices $\mathcal{U}^i_{\x}$ in the box $S_{\x}^n$ (see \eqref{unchecked}). As noted in item 3 of Remark~\ref{remark_1}, each execution of the growth algorithm checks at most $L(T-1)$ new vertices. Besides, verifying Conditions H2 and V2 requires checking at most $L$ vertices per box. Moreover, if at some later stage the exploration needs to inspect the renormalized vertex $S^n_{\x+(0,1)}$ through a vertical connection from $S^n_\x$, then an additional $L$ vertices must be checked in $S^n_\x$. Consequently, each exploration checks at most $L(T-1)+2L = L(T+1)$ vertices in $S^n_\x$.


Furthermore, as discussed in Remark~\ref{remark_1}, we have $T' \leq h$ and $T'' \leq h^2$. In both cases, this yields
\begin{equation}\label{BR}
    |\mathcal{U}^i_{\x} | \leq (i-1) L(T+1) \leq (h-1) L(h^2+1) \leq 2Lh^3, \mbox{ for all } i=1,\dots,h,
\end{equation}
with the convention that ${\mathcal{U}}^1_{\x} = \emptyset$.

Recall the ordered pairs of sets $(C_m^i,D_m^i)_{m\geq0}$ defined in Section~\ref{renor}. Assume $(C_{m-1}^i,D_{m-1}^i)$ has been constructed for some $m \geq 1$, and that $G_{\mathbf{y}}^j$ and $R_{\mathbf{y}}^j$ are defined for all $\mathbf{y}\in C_{m-1}^i$ and for all $1\leq j\leq i-1$.

If the $i$-th $\xi$-exploration reaches $\x$ at time $m$, there must be some $\mathbf{y}\in C_{m-1}^i$ with $\x=\mathbf{y}+(1,0)$ or $\x=\mathbf{y}+(0,1)$. We begin by considering the case $\x = \mathbf{y} + (1,0)$. 

Since $\y \in C_{m-1}^i$, we have $|G_{\y}^i| \geq h$. Consider the set
\begin{equation*}
    \Lambda_{\x}^i = \left\{ x \in S^n_{\x}\cap (\mathcal{U}^i_{\x})^c : \{ y,x\} \textrm{ is open for some } y \in G_{\y}^i \right\}.
\end{equation*}
Given $x \in S^n_{\x}\cap (\mathcal{U}^i_{\x})^c$ and $y \in G_{\y}^i$, it is clear that $|y-x| \leq 2n$.  By hypothesis, $p_i \geq (i \log i)^{-1}$ for all $i$ sufficiently large, thus we can assume that the probability that the edge $\{ y,x \}$ is declared open is bounded below by $(2n \log(2n))^{-1}$, except for finitely many vertices $x$ near the boundary of $S^n_{\x}$, where it may be that $p_{|x-y|} < [|x-y|\log(|x-y|)]^{-1})$. Specifically, the relevant exception is when $|x-y|<i_0$, where $i_0$ is the threshold beyond which $p_i\geq (i\log i)^{-1}$ holds. Since there are at most $O(i_0)$ such pairs per $y$, and since this is uniform in $n$, for $h$ large at least $h-O(1)\geq h/2$ of the $h$ vertices in $G^i_\y$ give the bound $p_{|x-y|}\geq (2n\log (2n))^{-1}.$   Thus, for every $n$ large,
\begin{equation}\label{Eqc}
    \P^{2n}_{p,\eps} \left( x \in \Lambda_{\x}^i   \right) \geq 1 - \left[ 1 - \frac{1}{2n \log(2n)} \right]^{h/2}  \geq 
    1 - \left[ 1 - \frac{1}{2n \log(2n)} \right]^{C\log(n)/2} 
    \geq \frac{c_1(C)}{n}, 
\end{equation}
where the constant $c_1=c_1(C) > 0$ can be made arbitrarily large as a function of $C$.  

The bounds obtained in \eqref{BR} for $\mathcal{U}^i_{\x}$ imply that given $c_2 \in (0,1)$, we have $|S^n_{\x}\cap(\mathcal{U}^i_{\x})^c| \geq c_2n$ for every large $n$. Hence, choosing $C$ (and consequently $c_1$) large enough, we obtain
\begin{equation}\label{EqC1}
    \P^{2n}_{p,\eps} \left( |\Lambda_{\x}^i| \geq L \right) \geq P\left( \textrm{Bin}\left( c_2n , \frac{c_1}{n} \right) \geq L \right) \geq (1 - \delta)^{1/3}, 
\end{equation}
for every $n$ large.

A similar argument applies for the case $\x = \mathbf{y} + (0,1)$. Define
\begin{equation}\label{upsilon}
    \Upsilon_{\x}^i = \left\{ x \in S^n_{\x}\cap (\mathcal{U}^i_{\x})^c : [x-(0,1)] \notin \mathcal{U}^{i}_{\y} \right\} ,
\end{equation}
\begin{equation*}
    \Delta_{\x}^i = \left\{ x \in \Upsilon_{\x}^{i} : \{ y,x-(0,1)\} \textrm{ and }\{ x-(0,1),x\}   \textrm{ are open for some } y \in G_{\y}^i \right\}.
\end{equation*} 
Observe that $\vert \Upsilon^i_\x\vert \geq n/2$. We proceed as in \eqref{Eqc} and \eqref{EqC1}, with only a minor modification: a factor of $\eps$ is introduced to account for the requirement that the vertical edge $\{ x-(0,1),x\}$ be open. Then, for each fixed $L$, by choosing $C$ sufficiently large, we obtain $\P^{2n}_{p,\eps} \left( |\Delta_{\x}^i| \geq L \right) \geq  (1 - \delta)^{1/3}$,  for every $n$ large.

We have shown that Conditions H1' and V1'  hold with arbitrarily large probability, independently of the past. To complete the proof, it remains to show the same for H2, V2 and C3'. This follows because, assuming H1' (or V1') holds, the probability that H2 (or V2) also holds is larger than $[1- (\min\{p,1-p\})^{2}]^L \geq (1-\delta)^{1/3}$, provided $L$ is chosen sufficiently large. 
Finally, assuming that H1' and  H2 (or V1' and V2) hold,  it suffices to show that the probability that the growth algorithm succeeds is also arbitrarily high. This follows directly from Lemma~\ref{successful}, as explained below.

Conditioned on the past and that H1' and  H2 (or V1' and V2) hold, the growth algorithm in $\x$ starts with a set $B_1$ satisfying $|B_1| =  L + |\mathcal{U}^i_{\x}|$. Using the bound in \eqref{BR}, we obtain  $|B_1| \leq 3Lh^3$.  Consequently, Lemma~\ref{successful} ensures that, for a suitable choice of $L$ and for every sufficiently large $n$, C3' is satisfied with probability larger than $(1-\delta)^{1/3}$. This completes the proof.
\end{proof}

Lemma~\ref{good_vert} shows that each potential exploration stochastically dominates an oriented site percolation model with parameter $1-\delta$. Hence, by applying  Proposition~\ref{aux2}, we can prove Lemma~\ref{lemmaboundchances}. 

\begin{proof}[Proof of Lemma~\ref{lemmaboundchances} ]
Recall the definitions in \eqref{Xifinite} and \eqref{projeciton} and the events in \eqref{Bproj} and \eqref{Bproj2}. Given $\ell\in\Z_+$ and $\xi \in \Xi_{\ell}$, fix some $\phi_{\xi} \in \Xi$ such that $\sigma_{\ell}(\phi_\xi) = \xi$. Write $$\Psi= \big[ \cap_{j< i} (\mathcal{B}_{\ell}^j(\xi))^c \big] \cap \mathcal{B}_{\ell-1}(\phi_{\xi}) = \big[ \cap_{j< i} (\mathcal{B}_{\ell}^j(\phi_{\xi}))^c \big] \cap \mathcal{B}_{\ell-1}(\phi_{\xi}),$$ 
and define the probability measure
$$\mathbb{Q}^K_{p,\eps}(A)=\P^K_{p,\eps}(A|\Psi),$$ for any measurable set $A$.

 Recall that $\mathcal{U}^i_{\x} \coloneq \cup_{r=1}^{i-1}  R_{\x}^r $, $i = 1, \dots, h$, with $ \mathcal{U}_{\x}^1 = \emptyset$, and define the sets $$\Delta' = \left\{ \big( \mathcal{U}_{\x}^i \big)_{\x \in T_{\ell}} : \mathcal{U}^i_{\x} \subset S^n_{\x} ~\textrm{ and }~ |\mathcal{U}^i_{\x}| \leq Lh^2(h -1) \right\},$$ 
 $$\Delta=\{(S,X): |S|\geq 4^{\ell-1}, X \in \Delta'\}.$$ 

Denote $\mathcal{\widetilde{U}}^i = \big( \mathcal{U}^i_{\x}\big)_{\x \in T_{\ell} }$, and consider the random vector $Y = (\Theta_{\ell},\widetilde{\mathcal{U}}^i)$. Observe that $\mathbb{Q}_{p,\eps}^K(\Delta)=1$. Given $a > 0$, let $\gamma(a)$ be as in Proposition~\ref{aux2}. For $C$ and $L$ as obtained in Lemma~\ref{good_vert}, taking $K=2n$,  apply Lemma \ref{good_vert} with $1-\delta > \gamma(a)$ to obtain 
\begin{align*}
    \P_{p,\eps}^K\left( \left(\mathcal{B}_{\ell}^i(\xi)\right)^c \Bigg\vert  \bigcap_{j< i} (\mathcal{B}_{\ell}^j(\xi))^c  \cap \mathcal{B}_{\ell-1}(\phi_{\xi}) \right) 
   &= \int_{\Delta}  \mathbb{Q}^K_{p,\eps} \left( \left(\mathcal{B}_{\ell}^i(\phi_{\xi})\right)^c \bigg\vert Y = (S,X)  \right) dF_{Y} (S,X) \\
   &= \int_{\Delta} \mathbb{Q}^K_{p,\eps}  \left(  |\Gamma_{\ell}^i(\phi_{\xi})| < 4^{\ell} \bigg\vert Y = (S,X)  \right) dF_{Y} (S,X)  \\
   &\leq \int_{\Delta} P_{\gamma}\left( |C_{S} \cap \overline{F}_{\ell+1}| < 4^{\ell} \right)  dF_{Y} (S,X) \\
   &\leq a^{4^{\ell}}
\end{align*}
for every $n$ sufficiently large. This completes the proof of Lemma~\ref{lemmaboundchances}.
\end{proof}

\section{Proof of Lemma~\ref{lemmaverdaorigem}}\label{proofoflemmas}


Recall Conditions S1, S2 and S3' from Definition \ref{seed_def}. To prove Lemma~\ref{lemmaverdaorigem}, we construct a configuration that contains a $(\xi,\ell_0)$-seed for every $\xi \in \Xi$. Before that, let us introduce further notation. 

Given $\ell_0\in\N$, consider the finite set $$\Lambda = \{v \in \Z^2 : v \in S_{\u}^n \textrm{ for some } \u \in \Z_+^2 \textrm{ such that } ||\u|| < 4^{\ell_0} \},$$ and let $(\omega\times\eta)_{\Lambda}$ denote the restriction of the configuration $(\omega\times\eta)$ to $\Lambda$. We denote by ${\P^K_{p,\eps}}_{\vert \Lambda}$ the restriction of $\P^K_{p,\eps}$ to $\Lambda$.


Assume, with no loss of generality, that $n$ is even. For each $\u=(\u_1, \u_2) \in \mathbb{Z}^2_+$, we split the one-dimensional horizontal box $S_{\u}^{n}$ into the sets
\begin{equation}\label{S-}
    S_{\u}^{n,-} = \{ y = (y_1, y_2) : y_1 \in [n\u_1, n\u_1 + n/2 -1], y_2 = \u_2\},
\end{equation}
\begin{equation}\label{S+}
    S_{\u}^{n,+} =  \{ y = (y_1, y_2) : y_1 \in [n\u_1 + n/2, n(\u_1 + 1) -1], y_2 = \u_2\}.
\end{equation}

We say that $A^{-}(\xi)$ occurs if, for every $\v \in \overline{F}_{\ell_0}$ and every $v \in S_{\v}^{n,-}$, $v$ is $(\xi,t(v))$-special with $t(v) \leq 2 \cdot 4^{\ell_0}$. The definition of $A^{+}(\xi)$ is analogous. We are ready to prove Lemma \ref{lemmaverdaorigem}.

\begin{proof}[Proof of Lemma~\ref{lemmaverdaorigem}]
Given $\ell_0\in\N$, let $m>0$ be such that $p_i>0$ for all $i\geq m$, and assume $n \geq 2m+6$. We will construct a configuration $\omega^{\ast}_{\Lambda} = (\omega'\times\eta')_{\Lambda}$ with ${\P^K_{p,\varepsilon}}_{\vert \Lambda} \big( \omega^{\ast}_{\Lambda}\big) > 0$, such that, for every $\xi \in \Xi$, either $A^{-}(\xi)$ or $A^{+}(\xi)$ occurs in $\omega^{\ast}_{\Lambda}$.

Recall Definition \ref{seed_def}. The configuration $\omega^{\ast}_{\Lambda}$ is such that $\overline{F}_{\ell_0}$ is a $(\xi,\ell_0)$-seed for every $\xi \in \Xi$. Since $\left\vert \overline{F}_{\ell_0} \right\vert = 4^{\ell_0}$, Condition S1 is satisfied. Given $C$, for every $n$ large such that $n/2 \gg h^2 = \lceil C\log(n)\rceil^2$, the set $S_{\v}^{n,+}$ (respectively,  $S_{\v}^{n,-}$) can be partitioned into $h$ subsets of size at most $h$. Since every $v \in S_{\v}^{n,+}$ (respectively, $v \in S_{\v}^{n,-}$) is  $(\xi,t(v))$-special for some $t(v) \leq 2 \cdot 4^{\ell_0}$, conditions S2 and S3' are also satisfied.

We build $\omega^{\ast}_{\Lambda}$ as follows. First, let $\omega'$ be such that  $\omega'(\{ u, u+(k,0) \}) = 1$ for all $u,\, u+(k,0) \in \Lambda$ with $p_k > 0$, and open all vertical edges inside $\Lambda$, i.e., $\omega'(\{ u, u+(0,1) \}) = 1$.  For this reason, we must choose the truncation constant such that $K \geq 4^{\ell_0}n$. 

First, we use the renormalized column $\{S^n_{\bf{(0,y)}} : 0 \leq y < 3\cdot4^{\ell_0-1}\}$ to see all words of length $2^{2(3\cdot4^{\ell_0-1}-1)}$. Next, we use a long-range edge to connect this column directly to the set $\overline{F_{\ell_0}}$. More precisely,
for each $0 \leq y < 3\cdot4^{\ell_0 - 1} $, write
\begin{equation}\label{conf1}
\eta'(x,y) = 
    \begin{cases}
    0, ~  \textrm{ if } x = m, m+1  \textrm{ and $y$ is even,}\\
    1, ~ \textrm{ if } x = m+2, m+3  \textrm{ and $y$ is even,}\\
    0, ~ \textrm{ if } x = m, m+2 \textrm{ and $y$ is odd,} \\
    1, ~ \textrm{ if } x = m+1, m+3  \textrm{ and $y$ is odd.}
    \end{cases}
\end{equation}    

For each $1 \leq y < 3\cdot4^{\ell_0 - 1}$, write
\begin{equation}\label{conf2}
\eta'(x,y) = 
    \begin{cases}
    0, ~  \textrm{ if } x = 2m+3, 2m+4  \textrm{ and $y$ is odd,}\\
    1, ~ \textrm{ if } x = 2m+5, 2m+6  \textrm{ and $y$ is odd,}\\
    0, ~ \textrm{ if } x = 2m+3, 2m+5 \textrm{ and $y$ is even,} \\
    1, ~ \textrm{ if } x = 2m+4, 2m+6  \textrm{ and $y$ is even.}
    \end{cases}
\end{equation}    

For each $\u \in \overline{F}_{\ell_0}$, set
\begin{equation}\label{conf3}
    \eta'(u) = 
    \begin{cases}
    0, ~ \textrm{ if } u \in S^{n,-}_{\u}, \\
    1, ~ \textrm{ if } u \in S^{n,+}_{\u},
    \end{cases}
\end{equation}
where $S^{n,-}_{\u}$ and $S^{n,+}_{\u}$ are defined in \eqref{S-} and \eqref{S+}, respectively. The states $\eta'(\cdot)$ of the remaining vertices of $\Lambda$ are irrelevant.

\begin{figure}[ht]
\centering

\begin{tikzpicture}[scale=0.70, every node/.style={transform shape}]
\draw[thick, dashed, black] (1/2,1) rectangle (1.65,4.85);
\foreach \y in {1,2,...,16}
{
\foreach \x in {1,...,\numexpr17-\y\relax}
{   
    \pgfmathparse{\x}
    \let\t\pgfmathresult
    \pgfmathparse{\x + \y}
            \let\sum\pgfmathresult
             \ifdim\sum pt>16pt  {
             \ifdim\t pt>4pt {
             \ifdim\t pt<13pt{
             
            \node [rectangle,draw,minimum width=1cm,minimum height=0.1cm, gray, fill = gray] at (1.1*\x, 1.1 + 0.3*\y){};
            }
            \else{\node [rectangle,draw,minimum width=1cm,minimum height=0.1cm, gray] at (1.1*\x, 1.1 + 0.3*\y) {};}
            \fi
            }
            \else{\node [rectangle,draw,minimum width=1cm,minimum height=0.1cm, gray] at (1.1*\x, 1.1 + 0.3*\y) {};}
            \fi
            }
            \else {
            \node [rectangle,draw,minimum width=1cm,minimum height=0.1cm, gray] at (1.1*\x, 1.1 + 0.3*\y) {};}
            \fi

}    
}

            \node [gray] at (9.5, 4.5){\large$\overline{F}_{\ell_0}$};

 \begin{scope}[shift={(9, 6)}, scale=1.5]
        \draw[thick, dashed, black] (0,0) rectangle (6,3.8);
    

        \foreach \y in {0,1,...,11} {
           
               \node [rectangle,draw,minimum width=5.9cm,minimum height=0.1cm, gray ] at (3 cm, 0.3 + 0.3*\y) {};

        }
        
              \foreach \y in {0,2, 4, 6, 8, 10} {  
              \draw (0.5 +0*0.2,0.3 + 0.3*\y) circle (1pt);
              \draw (0.5 +1*0.2,0.3 + 0.3*\y) circle (1pt);
              \fill (0.5 +2*0.2,0.3 + 0.3*\y) circle (1pt); 
              \fill (0.5 +3*0.2,0.3 + 0.3*\y) circle (1pt);  
               }

               \foreach \y in {0,2, 4, 6, 8, 10} {  
              \draw (0.5 +0*0.2,0.3 + 0.3*\y +0.3) circle (1pt);
              \fill (0.5 +1*0.2,0.3 + 0.3*\y+0.3) circle (1pt); 
              \draw (0.5 +2*0.2,0.3 + 0.3*\y +0.3) circle (1pt);
              \fill (0.5 +3*0.2,0.3 + 0.3*\y+0.3) circle (1pt);  
               }

                \foreach \y in {1,3, 5, 7, 9, 11} {  
              \draw (1.5 +0*0.2,0.3 + 0.3*\y) circle (1pt);
              \draw (1.5 +1*0.2,0.3 + 0.3*\y) circle (1pt);
              \fill (1.5 +2*0.2,0.3 + 0.3*\y) circle (1pt); 
              \fill (1.5 +3*0.2,0.3 + 0.3*\y) circle (1pt);  
               }

               \foreach \y in {1,3, 5, 7, 9} {  
              \draw (1.5 +0*0.2,0.3 + 0.3*\y +0.3) circle (1pt);
              \fill (1.5 +1*0.2,0.3 + 0.3*\y+0.3) circle (1pt); 
              \draw (1.5 +2*0.2,0.3 + 0.3*\y +0.3) circle (1pt);
              
              \fill (1.5 +3*0.2,0.3 + 0.3*\y+0.3) circle (1pt);  
               }
        
            \draw (0.5 +1*0.2,0.3 ) -- (0.5 +1*0.2,0.3 +0.3) to[out=20, in=160] (1.5 +3*0.2,0.3 + 0.3*1) circle (1pt) --
            (1.5 +3*0.2,0.3 + 0.3*1) --(1.5 +3*0.2,0.3 + 0.3*2) 
            to[out=160, in=20] (0.5 +1*0.2,0.3 +0.3 *2) -- (0.5 +1*0.2,0.3 +0.3 *3) to[out=20, in=160] (1.5 +0*0.2,0.3 + 0.3*3) circle (1pt)
            -- (1.5 +0*0.2,0.3 + 0.3*4) to[out=160, in=20] (0.5 +3*0.2,0.3 +0.3 *4) -- (0.5 +3*0.2,0.3 +0.3 *5) to[out=20, in=160] (1.5 +0*0.2,0.3 + 0.3*5) -- (1.5 +0*0.2,0.3 + 0.3*6) to[out=160, in=20] (0.5 +2*0.2,0.3 +0.3 *6) -- (0.5 +2*0.2,0.3 +0.3 *7) to[out=20, in=160] (1.5 +0*0.2,0.3 + 0.3*7) -- (1.5 +0*0.2,0.3 + 0.3*8) to[out=160, in=20] (0.5 +1*0.2,0.3 +0.3 *8) -- (0.5 +1*0.2,0.3 +0.3 *9) to[out=20, in=160] (1.5 +0*0.2,0.3 + 0.3*9)
            -- (1.5 +0*0.2,0.3 + 0.3*10) to[out=160, in=20] (0.5 +3*0.2,0.3 +0.3 *10) -- (0.5 +3*0.2,0.3 +0.3 *11);

    \end{scope}

    \draw[->, line width=0.4mm] (9, 8) to[out=150, in=60] (1,4.85);

\end{tikzpicture}

\caption{An illustration of the path $((x_1,0), (x_2,1), (x_3,1), (x_4,2), \dots, (x_{22},11))$ with respect to the finite word $\phi = (0, 1,1,1,0,1,0,0,1,1,0,0,1,0,0, 0 ,0 ,1,0,0,1,1)$ seen from the origin. From each $(x_{2j},j)$ we can use all direct connections of the form $\{(x_{2j},j),v\}$ to vertices $v \in S_{(\v_1,j)}^{n,-}$ or $v \in S_{(\v_1,j)}^{n,+}$.}
\label{colagem}
\end{figure}

The configuration is built so that, in the columns $\{m, m + 1, m + 2, m + 3\}$,  the labels are 0011 at even heights and 0101 at odd heights, while in the columns $\{2m+3,\dots,2m+6\}$ the pattern is reversed: the labels are 0011 at odd heights and 0101 at even heights. Consequently, this four-column structure can display any binary sequence of length $2y$ while moving upward through alternating column blocks.

Figure~\ref{colagem} illustrates why \eqref{conf1} -- \eqref{conf3} are sufficient for the occurrence of either $A^{+}(\xi)$ or $A^{-}(\xi)$ for every $\xi$, as explained below.

For each $ 0 \leq y < 3\cdot4^{\ell_0-1} $ and every sequence $\phi \in \{0,1\}^{2y}$ of length $2y$, there exists a path $(x_1,0), (x_2,1), (x_3,1), (x_4,2), \dots, (x_{2y},y)$, where $x_j \in \{m,\dots,m+3\}$ for $j \equiv  1 \textrm{ or }2\mod 4$, and $x_j \in \{2m+3,\dots,2m+6\}$ for $j \equiv 3 \textrm{ or } 0\mod 4$, such that $\eta'(x_j,\lfloor j/2 \rfloor) = \phi_j$, $j = 1, \dots, 2y$.

Given $\xi \in \Xi$, let $ \v = (\v_1, \v_2) \in \overline{F}_{\ell_0}$. If $\xi_{2\v_2+1} = 0$, then $A^{-}(\xi)$ occurs. Indeed, for each $v \in S_{\v}^{n,-}$, there exists a path $u_1, \dots, u_{2\v_2}, v$ such that $u_j = (x_j,\lfloor j/2 \rfloor)$, with $\eta'(u_j) = \xi_j$ for $j = 1, \dots, 2\v_2$ and $\eta'(v) = \xi_{2\v_2+1} = 0$.  That is, every $v \in S_{\v}^{n,-} $ is $(\xi,2\v_2+ 1)$-special. If $\xi_{2\v_2+1} = 1$, we conclude that every $v \in S_{\v}^{n,+} $ is $(\xi,2\v_2+ 1)$-special. Finally, since $ \v = (\v_1, \v_2) \in \overline{F}_{\ell_0}$, we obtain $2\v_2+1 < 2\cdot4^{\ell_0}$.
\end{proof}

\section{Proof of Lemma~\ref{lemmaentropy} }\label{proofofentropy}

Recall the sequence $s(\ell) = 2\cdot 4^{\ell}h$, $\ell \in \N$. For $\xi \in \Xi$, Lemma~\ref{lemmaentropy} provides bounds on the entropy generated by the exploration processes that construct paths of $\xi$-special vertices. As noted in the proof of Theorem~\ref{teo1}, this control is essential to restrict the analysis to finite words and apply the union bound.

\begin{proof}[Proof of Lemma~\ref{lemmaentropy}]
    Recall that $\Gamma_{\ell}^i(\xi)$ denotes the set of $i$-good renormalized vertices in $\overline{F}_{\ell+1}$ obtained in the $i$-th $\xi$-exploration process in the slice $T_{\ell}$ (see \eqref{successful2} and \eqref{successful3}). 

    Given $\ell_0 \in \N$, assume that for some $\ell \geq \ell_0$ and some $i \in \{1, \dots, h\}$, the $i$-th $\xi$-exploration in $T_{\ell}$ is executed, inducing the set $\Gamma_{\ell}^i(\xi)$.  The $\xi$-explorations start from a $(\xi, \ell_0)$-seed, which we denote by $\Theta_{\ell_0}$.  Given $\z \in \Gamma_{\ell}^i(\xi)$, there exists an oriented path of good renormalized vertices $(\z_{4^{\ell_0}}, \dots, \z_{4^{\ell+1}} )$ with $\z_{4^{\ell_0}} \in \Theta_{\ell_0}$ and $\z_{4^{\ell+1}} = \z$, such that
 Conditions H1', H2 (or V1', V2), and C3' are satisfied with $\y = \z_{j-1}$ and $\x = \z_{j}$, for each $j = 4^{\ell_0}+1, \dots, 4^{\ell+1}$.

 Denoting by $x_j \in S_{\z_j}^n$ and $y_{j-1} = y(x_j) \in S_{\z_{j-1}}^n$ the corresponding vertices $x \in S_{\x}^n$ and $y(x) \in S_{\y}^n$ derived from Conditions H1', H2 (or V1' and V2) and C3', we have $t(x_j) \leq t(y_{j-1}) + 2$. By Remark~\ref{remark_1},
\[
t(y_{j}) <
\begin{cases}
t(x_{j}) + h^2, & \text{if } j \in \{4^{\ell_0+1},\dots,4^{\ell}\},\\[0.3em]
t(x_{j}) + h, & \text{otherwise}.
\end{cases}
\]
Therefore, we have the inequality
\[
t(y_{j}) <
\begin{cases}
t(y_{j-1}) + h^2 + 2, & \text{if } j-1 \in \{4^{\ell_0+1},\dots,4^{\ell}\},\\[0.3em]
t(y_{j-1}) + h + 2, & \text{otherwise}.
\end{cases}
\]

Following the notation above, each microscopic vertex $v \in G_{\z}^i = G_{\z_{4^{\ell+1}}}^i$ arises from the $i$-th execution of the growth algorithm in $S^n_{\z}$ starting at $x_{4^{\ell+1}}$. By Remark~\ref{remark_1} (item 2), $t(v) < t(x_{4^{\ell+1}}) + h^2$ for all $v \in G_{\z}^i$.

From Condition S3', we have $t(y_{4^{\ell_0}}) \leq 2\cdot4^{\ell_0}$. Combining this with the previous bounds, we obtain for every  $v \in G_{\z}^i$,
\begin{align*}
    t(v) < t(x_{4^{\ell+1}}) + h^2 &\leq t(y_{4^{\ell+1}-1}) + 2 + h^2 \\
    &< t(y_{4^{\ell_0}}) + (4^{\ell+1} - 4^{\ell_0})(h + 2) + (\ell - \ell_0)(h^2 + 2) \\
    &< 2\cdot 4^{\ell_0} + 4^{\ell+1}(h + 2) + \ell(h^2 + 2) \\
    &\leq 4^{\ell+1}h + [2\cdot 4^{\ell} + 2\cdot4^{\ell+1} + \ell(h^2 + 2)].
\end{align*}

Since $h = \lceil C\log(n)\rceil$, for each fixed $C$ and sufficiently large $n$, one can choose $\ell_0(C,n)$ such that 
\begin{equation*}
    t(v) < 2\cdot4^{\ell+1}h,  \quad \forall v \in G_{\z}^i,
\end{equation*}
for all $\ell \geq \ell_0$,  as claimed. 
\end{proof}

\section{Generalization}\label{remarkfinal}

The conclusion of Theorem \ref{teo1} remains true if condition \eqref{log} is relaxed. Specifically, we have the following theorem.
\begin{theorem}\label{teogen} Let $(q_n)_{n\in\N}$ be a non-increasing sequence such that $nq_{2n}\rightarrow 0$ as $n$ goes to infinity, and let $\beta(n)=(nq_{2n})^{-3}$. Assume that
 \begin{itemize}
     \item[(i)] for some $i_0 \in \N$,  $p_i \geq q_i$ for all $i \geq i_0$; 
     \item[(ii)]  for every $c > 0$, $$\sum_{i = \lceil c\beta(n) \rceil }^{n/2}  q_i \longrightarrow \infty,\quad\quad\mbox{ as $n\rightarrow \infty$}.$$
 \end{itemize}
  Then, for all $p\in(0,1)$ and $\varepsilon>0$, there exists $K=K(\varepsilon, p,(p_i))\in\N$ such that $$\P_{p,\eps}^K(W_0 =  \Xi )>0.$$
\end{theorem}

 We do not present a detailed proof of Theorem~\ref{teogen}, as it can be obtained by a straightforward adaptation of the arguments used in the proof of Theorem~\ref{teo1}.
 Indeed, compared to \eqref{h}, it suffices to set $h(n) = \lceil C(
 nq_{2n})^{-1}\rceil$ for a suitably large $C$ depending on $L$, where $L$ is chosen in terms of $p$ and $\varepsilon$, as described earlier. 
 In particular, if we want the asymptotic estimates used in the proof of Theorem~\ref{teo1} to remain valid, then the condition $h(n)\to\infty$ is necessary, which in turn implies that $nq_{2n}\to 0$.
 Additionally, the bound in \eqref{BR} implies that, in each run of the growth algorithm, $|B_1| \leq \tilde{c}(h(n))^3< c\beta(n)$, for some positive constants $\tilde{c}$ and $c$. Consequently, assumption (ii) above, similarly to \eqref{divergence}, yields a result analogous to Lemma~\ref{successful}. 
 
 Finally, this choice of $h(n)$ ensures properties analogous to those obtained in \eqref{Eqc} and \eqref{EqC1} (the key ingredients of Lemma~\ref{good_vert}). Particularly, the last term in \eqref{Eqc} becomes \[ 1 - \left[ 1-  {q_{2n}} \right]^\frac{h}{2} \geq  1 - \left[ 1-  {q_{2n}} \right]^\frac{C}{2nq_{2n}} \geq \frac{c_1}{n}.\]

As an example, let $\log_{(1)}(x) = \log x$ and $\log_{(m+1)}(x) = \log \log_{(m)}(x)$ for $m \in \N$. For any $m \in \N$ and $C > 0$, the conditions above are satisfied if  
 \begin{equation*}
     q_i = \frac{C}{i\log_{(1)}(i)\log_{(2)}(i)\dots\log_{(m)}(i)}.
 \end{equation*}
In this case, percolation on the truncated lattice was proved in \cite{MSV}.

\section{Proof of Theorem~\ref{teo2}}\label{sec_3}

In this section, we prove Theorem \ref{teo2}. The proof proceeds as follows. First, we show that our graph is isomorphic to a three-dimensional slab with one-dimensional long-range connections. Next, following \cite{GLS}, we construct a dynamical coupling between percolation on this slab and an independent nearest-neighbor oriented percolation process with parameter $\gamma\in(0,1)$, typically large. This coupling implies that all words are seen in the slab, and hence in the original graph.

\subsection{The isomorphism}\label{iso}

Let $\mathbb{S}_k$ denote a three-dimensional slab of thickness $k+1$ with one-dimensional long-range connections. More precisely, let $(\mathbf{e}_1, \mathbf{e}_2,\mathbf{e}_3)$ be the canonical basis of $\mathbb{R}^3$, and define the sets 
$$\E^{sl}_V = \cup_{n=1}^k \E^{sl}_{V,n}, ~\textrm{ where } ~ \E^{sl}_{V,n} =\{\{ u,u+n\mathbf{e}_3\} : u\in \Z^3\}, \quad 1\leq n\leq k,$$
$$\E^{sl}_{H,i}=\{\{ u,u+\mathbf{e}_i\} : u\in \Z^{3}\}, \quad i=1,2,$$
$$\E^{sl}= \E^{sl}_V\cup \E^{sl}_{H,1}\cup  \E^{sl}_{H,2}.$$
Write $\mathbb{V}=\Z^2\times\{0,\dots,k\}$ and denote $\mathbb{S}_k=(\mathbb{V},\E^{sl})$.

Let $G_K=\left(\Z^2,\E_{V}\cup\left(\bigcup_{i=1}^K\E_{H,i}\right)\right)$ be the $K$-truncation of the graph introduced in Section \ref{intro}, see \eqref{vertical} and \eqref{horizontal}. Consider the following subset of edges of $G_K$
$$\mathcal{W}(K) = \left\{\{(u_1,u_2),(v_1,v_2)\} \in \mathbb{Z}^2 \times \mathbb{Z}^2 : u_2 = v_2,  \lceil u_1/K \rceil = \lceil v_1/K \rceil \right\} \cup \mathcal{E}_{H,K} \cup \mathcal{E}_{V},
$$
and define the graph $F_K=\left(\Z^2 , \mathcal{W}(K) \right)$. 

We claim that the graphs $F_K$ and $\mathbb{S}_K$ are isomorphic (see also Section 3 of \cite{LSS}). To see this, define the function
$$
\varphi: \mathbb{Z} \rightarrow \mathbb{Z} \times \{0,1,\dots, K-1\}, \quad \textrm{where } ~~~\varphi(u) :=\left(\left \lceil \frac{u}{K}\right \rceil, u \mod K \right).
$$
Indeed, the function
$$
\Phi: \mathbb{Z} \times \mathbb{Z} \rightarrow \mathbb{Z}^2 \times \{0,1,\dots, K-1\},$$
defined by $\Phi(u, v) :=(\varphi(u), v),$
is a graph isomorphism between $F_K$ and $\mathbb{S}_K$; see Figure \ref{isomorphism}. We observe that edges in $\E_{H,K}$ are identified with edges in $\E^{sl}_{H,1}$, edges in $\E_{H,j}$ are identified with edges in $\E^{sl}_{V,j}$, $1\leq j\leq K-1$, and edges in $\E_{V}$ are identified with edges in $\E^{sl}_{H,2}$. 


\begin{figure}[ht]
\centering
\begin{tikzpicture}[scale = 0.9]
    
     \tznode(0,0){\textbf{$0$}}[b]
    \tzdot*(0,0)
    \tznode(1,0){\textbf{$1$}}[b]
    \tzdot*(1,0)
    \tznode(2,0){\textbf{$2$}}[b]
    \tzdot*(2,0)
   \tznode(3,0){\textbf{$3$}}[b]
    \tzdot*(3,0)
    \tznode(4,0){\textbf{$4$}}[b]
    \tzdot*(4,0)
    \tznode(5,0){\textbf{$5$}}[b]
    \tzdot*(5,0)
   \tznode(6,0){\textbf{$6$}}[b]
    \tzdot*(6,0)
   \tznode(7,0){\textbf{$7$}}[b]
    \tzdot*(7,0)
    \tznode(8,0){\textbf{$8$}}[b]
    \tzdot*(8,0)
    \tznode(9,0){\textbf{$9$}}[b]
    \tzdot*(9,0)
    \tznode(10,0){\textbf{$10$}}[b]
    \tzdot*(10,0)
    \tznode(11,0){\textbf{$11$}}[b]
    \tzdot*(11,0)
    \tznode(12,0){\textbf{$12$}}[b]
    \tzdot*(12,0)

    \tzto[gray, very thick] (0,0)  [below,near end] (2,0)
     \tzto[dashed, bend left = 25] (0,0)  [below,near end] (2,0)
     \tzto[bend left = 30, very thick] (0,0)  [below,near end] (3,0)
     \tzto[bend left = 30, very thick] (1,0)  [below,near end] (4,0)
     \tzto[bend left = 30, very thick] (2,0)  [below,near end] (5,0)

     \tzto[gray, very thick] (3+0,0)  [below,near end] (3+2,0)
     \tzto[dashed, bend left = 25] (3+0,0)  [below,near end] (3+2,0)
     \tzto[bend left = 30, very thick] (3+0,0)  [below,near end] (3+3,0)
     \tzto[bend left = 30, very thick] (3+1,0)  [below,near end] (3+4,0)
     \tzto[bend left = 30, very thick] (3+2,0)  [below,near end] (3+5,0)

   \tzto[gray, very thick] (6+0,0)  [below,near end] (6+2,0)
     \tzto[dashed, bend left = 25] (6+0,0)  [below,near end] (6+2,0)
     \tzto[bend left = 30, very thick] (6+0,0)  [below,near end] (6+3,0)
     \tzto[bend left = 30, very thick] (6+1,0)  [below,near end] (6+4,0)
     \tzto[bend left = 30, very thick] (6+2,0)  [below,near end] (6+5,0)

    \tzto[gray, very thick] (9+0,0)  [below,near end] (9+2,0)
     \tzto[dashed, bend left = 25] (9+0,0)  [below,near end] (9+2,0)
     \tzto[bend left = 30, very thick] (9+0,0)  [below,near end] (9+3,0)
     \tzto[bend left = 30, very thick] (9+1,0)  [below,near end] (9+4,0)
     
     \tzto[gray, very thick] (12+0,0)  [below,near end] (13,0)

\end{tikzpicture}
\vspace{0.8cm}

  \begin{tikzpicture}[scale = 1.1]
    \tzto[gray, very thick] (0,0)  [below,near end] (0,2);
    \tzto[gray, very thick] (1,0)  [below,near end] (1,2);
    \tzto[gray, very thick] (2,0)  [below,near end] (2,2);
    \tzto[gray, very thick] (3,0)  [below,near end] (3,2);
    \tzto[gray, very thick] (4,0)  [below,near end] (4,2);

    \tzto[dashed, bend left = 60] (0,0)  [below,near end] (0,2);
    \tzto[dashed, bend left = 60] (1,0)  [below,near end] (1,2);
    \tzto[dashed, bend left = 60] (2,0)  [below,near end] (2,2);
    \tzto[dashed, bend left = 60] (3,0)  [below,near end] (3,2);
    \tzto[dashed, bend left = 60] (4,0)  [below,near end] (4,2);

    \tzto[black, very thick] (0,0)  [below,near end] (4,0);
    \tzto[black, very thick] (0,1)  [below,near end] (4,1);
    \tzto[black, very thick] (0,2)  [below,near end] (4,2);

    \tznode(0,0){\textbf{$0$}}[b]
    \tzdot*(0,0)
    \tznode(0,1){\textbf{$1$}}[bl]
    \tzdot*(0,1)
    \tznode(0,2){\textbf{$2$}}[a]
    \tzdot*(0,2)
   \tznode(1,0){\textbf{$3$}}[b]
    \tzdot*(1,0)
    \tznode(1,1){\textbf{$4$}}[bl]
    \tzdot*(1,1)
    \tznode(1,2){\textbf{$5$}}[a]
    \tzdot*(1,2)
   \tznode(2,0){\textbf{$6$}}[b]
    \tzdot*(2,0)
   \tznode(2,1){\textbf{$7$}}[bl]
    \tzdot*(2,1)
    \tznode(2,2){\textbf{$8$}}[a]
    \tzdot*(2,2)
    \tznode(3,0){\textbf{$9$}}[b]
    \tzdot*(3,0)
    \tznode(3,1){\textbf{$10$}}[bl]
    \tzdot*(3,1)
    \tznode(3,2){\textbf{$11$}}[a]
    \tzdot*(3,2)
    \tznode(4,0){\textbf{$12$}}[b]
    \tzdot*(4,0)
    \tznode(4,1){\textbf{$13$}}[bl]
    \tzdot*(4,1)
    \tznode(4,2){\textbf{$14$}}[a]
    \tzdot*(4,2)

    \end{tikzpicture}    
\caption{The map $\Phi $ for $K=3$.} 
\label{isomorphism}
\end{figure}

\subsection{The coupling}\label{coupling}


Given $K\in\N$, let $\eps>0$ and $(p_{K,i})_{i\in\N}$ be as in \eqref{trun_seq}.  Consider the following percolation model on $\mathbb{S}_K=(\mathbb{V},\mathcal{E}^{sl})$. Edges in $\E_{H,1}^{sl}\cup \E_{H,2}^{sl}$ are independently open with probability $\eps$ and closed with probability $1-\eps$, while edges in $\E_{V,j}^{sl}$ are independently open with probability $p_j$ and closed with probability $1-p_j$.  Also, to each vertex $v\in\mathbb{V}$, we assign independently a random variable $\eta(v)$ taking the values 1 and 0 with probabilities $p$ and $1-p$, respectively. We refer to this model as the $SLAB(K,\varepsilon)$ model.

To show that all words are seen from the origin in the $SLAB(K,\varepsilon)$ model, we construct a coupling between the long-range process on the slab and an independent, highly supercritical nearest-neighbor oriented percolation process. This approach is similar to that of \cite{GLS}, where a long-range percolation process on $\LL^3$ is coupled with two-dimensional oriented percolation. In our case, however, the restriction to a slab requires a slight modification of the argument, which we present here for completeness.

Given a set $A\subset \Z^2_+$, write $\partial_e A$ for the external vertex boundary of $A$, defined by
$$\partial_e A=\{u\in A^c\colon  \exists v\in {A}\mbox{ such that } ||u-v||=1\}.$$
Fix $\xi\in\Xi$ and let $x_1,x_2,\dots$ be a fixed ordering of the vertices of $\LL^2_+$. We shall inductively construct a sequence $\{A_n,B_n\}_{n\geq 0}$ of ordered pairs of subsets of $\Z_+^2$ and a function $\psi:A_n\longrightarrow \Z_+.$ Write $(0,0)$ for the origin of $\LL^2_+$ and set $A_0=\{(0,0)\}$, $B_0=\emptyset$, and $\psi((0,0))=0$. 

Assume $\{A_n,B_n\}$ has been constructed for some $n\in\N$ and that $\psi(x)$ is known for all $x\in A_n$. If $\partial_e A_n\cap B_n^c=\emptyset$, we stop the construction and set $(A_{\ell},B_{\ell})=(A_n,B_n)$ for all $\ell\geq n$. If $\partial_e A_n\cap B_n^c\neq\emptyset$, let $x_n$ be the earliest vertex in the fixed ordering in $\partial_e A_n\cap B_n^c$, and define $y_n$ as the vertex in $A_n$ such that $x_n=y_n+(1,0)$ or $x_n=y_n+(0,1)$. First, assume that $x_n=y_n+(1,0)$. Fixing $N,M\in\N$, we distinguish between two cases:
\begin{enumerate}
\item[1.] $\psi(y_n)\leq N+M$

We say that $x_n$ is \textit{black} if, for some $i\in\{1,\dots, N\}$, the following conditions hold:
\begin{itemize}
\item $\eta(y_n,\psi(y_n)+i)=\xi_{2||y_n||+1}$,\, $(y_n,\psi(y_n)+i)\in \Z_+^2\times \Z_+,$
\item $\eta(x_n,\psi(y_n)+i)=\xi_{2||y_n||+2}$,\, $(x_n,\psi(y_n)+i)\in \Z_+^2\times \Z_+$,
\item $\{ (y_n,\psi(y_n)),(y_n,\psi(y_n)+i)\} \in \E_V^{sl}$, and $\{ (y_n,\psi(y_n)+i),(x_n,\psi(y_n)+i)\} \in \E_H^{sl}$ are open.
\end{itemize}
In this case, write $\psi(x_n)=\psi(y_n)+i$.
\item[2.] $\psi(y_n)> N+M$

We say $x_n$ is \textit{black} if, for some $i\in\{1,\dots, N\}$, the following conditions hold:
\begin{itemize}
\item $\eta(y_n,\psi(y_n)-i)=\xi_{2||y_n||+1}$, $(y_n,\psi(y_n)-i)\in \Z_+^2\times \Z_+,$
\item $\eta(x_n,\psi(y_n)-i)=\xi_{2||y_n||+2}$, $(x_n,\psi(y_n)-i)\in \Z_+^2\times \Z_+$,
\item $\{ (y_n,\psi(y_n)),(y_n,\psi(y_n)-i)\} \in \E_V$, and $\{ (y_n,\psi(y_n)-i),(x_n,\psi(y_n)-i)\} \in \E_H$ are open.
\end{itemize}
In this case, write $\psi(x_n)=\psi(y_n)-i$.
\end{enumerate}

Define
\begin{equation*}(A_{n+1},B_{n+1})=\left\{\begin{array}{ll}
(A_n\cup\{x_n\},B_n)&\mbox{if}\quad x_n\mbox{ is black},\\
 (A_n,B_n\cup\{x_n\})&\mbox{otherwise}.
\end{array}\right.
\end{equation*}

To handle the case $x_n=y_n+(0,1)$, we proceed analogously with the only difference that, when declaring $x_n$ black, $i$ is picked in the set $\{N+1,\dots,N+M\}$. This distinction is useful to avoid dependence issues. Writing $$A(\xi)=\bigcup_{n\in\N}A_n,$$ it follows that if $|A(\xi)|=\infty$, then $\xi$ is seen in the $2(N+M)$-truncated model. Note that
\begin{equation*}\P_{p,\eps}^{2N+2M}(x_n\mbox{ is not black})\leq\left\{\begin{array}{ll}
\prod_{i=1}^N[1-\varepsilon p_i(\min\{p,1-p\})^2]&\mbox{if}\quad x_n=y_n+(1,0),\\
 \prod_{i=N+1}^{N+M}[1- \eps p_i(\min\{p,1-p\})^2]&\mbox{if}\quad x_n=y_n+(0,1).
\end{array}\right.
\end{equation*}
In any case, if $\limsup p_i>0$, then the sum of the $p_i$ diverges and we get
\begin{equation}\label{black}
\P_{p,\eps}^{2N+2M}(x_n\mbox{ is black})\longrightarrow 1,\mbox{ when $N,M\rightarrow \infty$},
\end{equation}
for all $n\in\N$. Using \eqref{black}, Lemmas 1 and 2 in \cite{GLS} follow in a rather analogous way, allowing us to prove that, for all $p\in(0,1)$, $\varepsilon>0$, and $\alpha>0$, there exists $N=N(p,\varepsilon,\alpha)$ and $M=M(p,\varepsilon,\alpha)$ such that all words are seen simultaneously from the origin in the $SLAB(2N+2M,\varepsilon)$ model with probability larger than $1-\alpha$.

\begin{proof}[Proof of Theorem \ref{teo2}.]
Fix $\varepsilon$, $\alpha$, and $p$. By hypothesis, there exists $\delta>0$ such that $\limsup p_i>\delta$. Recall from Section \ref{iso} that edges in $\E_{H,K}$ are identified with those in $\E_{H,1}^{sl}$. Hence, by the discussion in Sections \ref{iso} and \ref{coupling}, there exists $K_1=K_1(p,\delta,\alpha)$ such that all words are seen from the origin in the $SLAB(K_1,\delta\wedge\varepsilon)$ model with probability larger than $1-\alpha$.
Next, pick $K_2\geq K_1$ such that $p_{K_2}>\delta$, noting that this choice is possible since $\limsup p_i>\delta$. Clearly, since $\mathbb{S}_{K_1}$ is a subgraph of $\mathbb{S}_{K_2}$, all words are seen from the origin in the $SLAB(K_2,\delta\wedge\varepsilon)$ model with probability larger than $1-\alpha$. The result follows by observing that $\mathbb{S}_{K_2}$ is isomorphic to $F_{K_2}$, which is itself a subgraph of $G_{K_2}$.
\end{proof}

\section*{Acknowledgements} The research of Pablo Gomes was partially supported by FAPESP, grant 2023/13453-5 and grant 2025/27064-6. Ot\'avio Lima was partially supported by CAPES and FAPESP, grant 2025/04550-2. Roger Silva was partially supported by FAPEMIG, grant APQ-06547-24. The authors are grateful to the referee for a careful evaluation and helpful comments, which contributed to improving this manuscript.




\end{document}